\documentclass[USenglish]{article}
\usepackage[utf8]{inputenc}%(only for the pdftex engine)

\usepackage{amssymb,amsmath,amsthm,amscd,epsf,latexsym,verbatim,graphicx,amsfonts,epstopdf,xcolor,wasysym,color,comment}
\input epsf.tex
\usepackage[utf8]{inputenc}
\usepackage[T1]{fontenc}
\usepackage{graphicx}
\usepackage{authblk}
\usepackage{caption}
\usepackage{subcaption}
\usepackage{palatino, url, multicol}
\usepackage{geometry}% See geometry.pdf to learn the layout options. There are lots.
\usepackage[linktoc=all]{hyperref}
\theoremstyle{theorem}
\newtheorem{theorem}{Theorem}[section]

\newtheorem{lemma}[theorem]{Lemma}
\newtheorem{proposition}[theorem]{Proposition}
\newtheorem{corollary}[theorem]{Corollary}

\newtheorem*{thmA}{Theorem \ref{thm:adj}}
\newtheorem*{thmB}{Theorem \ref{thm:hearing_angles}}
\newtheorem*{thmC}{Theorem \ref{thm:no_finite}}

\theoremstyle{definition}
\newtheorem{observation}[theorem]{Observation}
\theoremstyle{definition}

\theoremstyle{definition}
\newtheorem{remark}[theorem]{Remark}
\theoremstyle{definition}
\newtheorem{example}[theorem]{Example}
\theoremstyle{definition}
\newtheorem{definition}[theorem]{Definition}
\theoremstyle{definition}

\theoremstyle{question}
\newtheorem{question}[theorem]{Question}
\theoremstyle{question}

\theoremstyle{definition}

\def\L{\mathcal{L}}
\def\B{{\sf B}}

\def\N{{\mathbb N}}
\def\C{{\mathbb C}}
\def\R{{\mathbb R}}
\def\Z{{\mathbb Z}}
\def\D{{\mathcal D}}
\def\A{{\mathcal A}}

\def\ins{\operatorname{in}}

\newcommand{\set}[1]{ { \{ #1 \} } }

\makeatletter
\let\@fnsymbol\@arabic
\makeatother

\title{How to hear the shape of a billiard table}

\author{Aaron Calderon\footnote{Department of Mathematics, Yale University, 10 Hillhouse Avenue, New Haven, CT 06511, \textsf{aaron.calderon@yale.edu}} , Solly Coles\footnote{School of Mathematics, University of Bristol, Senate House, Tyndall Avenue, Bristol, BS8 1TH, \textsf{sc14367.2014@my.bristol.ac.uk }} , Diana Davis\footnote{Department of Mathematics and Statistics, Swarthmore College, 500 College Avenue, Swarthmore, PA 19081, \textsf{ddavis3@swarthmore.edu}} , Justin Lanier\footnote{School of Mathematics, Georgia Institute of Technology, 686 Cherry Street, Atlanta, GA 30332,  \textsf{jlanier8@gatech.edu}} , Andre Oliveira\footnote{Department of Mathematics and Computer Science, Wesleyan University, 45 Wyllys Avenue, Middletown, CT 06459, \textsf{aoliveira@wesleyan.edu }}}

\begin{document}

\maketitle

\begin{abstract}
The bounce spectrum of a polygonal billiard table is the collection of all bi-infinite sequences of edge labels corresponding to billiard trajectories on the table. We give methods for reconstructing from the bounce spectrum of a polygonal billiard table both the cyclic ordering of its edge labels and the sizes of its angles. We also show that it is impossible to reconstruct the exact shape of a polygonal billiard table from any finite collection of finite words from its bounce spectrum.
\end{abstract}

\tableofcontents

%%%%%

\section{Introduction}

In this paper we show how to recover geometric information about a polygonal billiard table from the symbolic dynamics of its billiard flow. This can be interpreted as a spectral rigidity result, in the same spirit as the question “Can one hear the shape of a drum?” asked by Kac in a classic 1966 paper \cite{drum}. We give a selection of related results within this tradition in \S 1.1.

Let $P$ be a polygonal billiard table whose edges are labeled by an alphabet $\A$. Given a bi-infinite nonsingular basepointed billiard trajectory $\tau$, let the corresponding {\em bounce sequence} $\B(\tau)$ be the $\Z$-indexed sequence of labels of the edges that the trajectory hits. The {\em bounce spectrum} $\B(P)$ of the polygon $P$ is the set of all sequences $\B(\tau)$ where $\tau$ is such a trajectory on $P$. The set of finite subwords appearing $\B(P)$ is the {\em bounce language} $\L_P$ of the polygon. Additional details about this setup are given in \S 2.

The motivating questions of this paper are: to what extent does $\B(P)$ determine $P$? What geometric information about $P$ can be reconstructed from $\B(P)$? Or more fancifully: what properties of $P$ can we ``hear" in $\B(P)$, and how can we go about doing so?

Our first two main results are that the adjacency of edges in $P$ and the sizes of the angles of $P$ can be reconstructed from $\B(P)$.

\begin{thmA} The adjacency of edges in a polygonal billiard table $P$ can be reconstructed from $\B(P)$.
\end{thmA}

\begin{thmB} The angles of a polygonal billiard table $P$ can be reconstructed from $\B(P)$.
\end{thmB}

In other words, we show how these pieces of geometric data are encoded within the uncountable collection of symbolic information recorded in $\B(P)$.

Our results are complementary to those in the recent paper by Duchin, Erlandsson, Leininger, and Sadanand \cite{DELS}. They show that $\B(P)$ is a complete invariant for $P$, up to elementary qualifications, as long as a cyclic labeling of the edges of $P$ is fixed. (A shared cyclic labeling is a standing assumption throughout their paper.)

\begin{theorem}[Duchin--Erlandsson--Leininger--Sadanand, \cite{DELS}, Bounce Theorem]\label{thm:DELS}
  If two simply connected polygons have the same bounce spectrum, then either they are similar such that the similarity respects the edge labeling, or they are both right-angled and related by an affine transformation.
\end{theorem}

Our Theorem 4.10 eliminates the need for their assumption that the polygons have a shared cyclic edge labeling, since it shows that this information can be derived from $\B(P)$. Additionally, while their result ensures that $\B(P)$ faithfully encodes the geometry of $P$, our results draw out exactly how the adjacency and angle information is encoded in $\B(P)$.

Our results about adjacency and angles give a method for reconstructing $P$ from $\B(P)$ up to a {\em parallel family} of polygons, an example of which is illustrated in Figure \ref{fig:parallel-pent}. It remains an open problem to produce a method for recovering the edge lengths of $P$ from $\B(P)$. We briefly discuss this problem in \S\ref{sec:lengths}. Producing a method for recovering lengths would provide, when combined with our Theorems \ref{thm:adj} and \ref{thm:hearing_angles}, an independent and constructive proof of Theorem \ref{thm:DELS}. Since triangles are determined up to similarity by their angles, our results do provide an independent and constructive proof of Theorem \ref{thm:DELS} in the case of triangles.

\begin{figure}[h]
  \centering
  \includegraphics[width=2in]{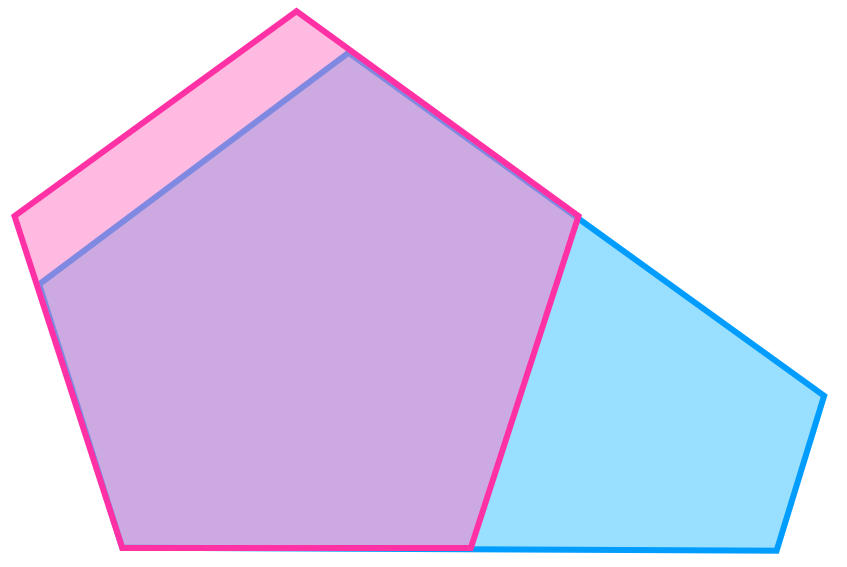}
  \caption{Overlapping polygons whose angles are all $3\pi/5$. They are in the same parallel family, so we need to know their edge lengths to distinguish between them.}
  \label{fig:parallel-pent}
\end{figure}

Our proofs of Theorems \ref{thm:adj} and \ref{thm:hearing_angles} crucially leverage the existence of certain bounce words of arbitrary length to reconstruct adjacency of edges and sizes of angles. Our third main result shows that one can never fully recover the shape of a polygon from a set of bounce words of bounded length.

\begin{thmC} A polygon $P$ cannot be reconstructed from any finite subset of its bounce language $\L_P$.
\end{thmC}

Despite this result, we point out that no matter the angle at the vertex between adjacent edges $A$ and $B$, it is possible to obtain a rough bound on its size using the lengths of strings of alternating $A$s and $B$s that appear within $\L_p$. Some examples are given in Table \ref{table:angles}. These calculations come from unfolding the corner of the polygon until copies of the unfolded angle sum to more than $\pi$, and then counting how many times a single line can cut across the unfolded edges. This observation is in fact the starting point of our result on reconstructing angles from $\B(P)$. The fact that finite words do give rough information about sizes of angles raises the following effectivization problem.

\begin{table}[h]
  \begin{center}
\begin{tabular}{  c | c  }
maximum length of word of alternating $A$s and $B$s & indicate that $\theta$ satisfies   \\ \hline
1 & $\pi \leq \theta$ \\
2 & $\pi/2 \leq \theta < \pi\phantom{/1}$ \\
3 & $\pi/3 \leq \theta < \pi/2$ \\
4 & $\pi/4 \leq \theta < \pi/3$ \\
\vdots& \vdots
\end{tabular}
\caption{Maximum lengths of strings of alternating $A$s and $B$s yield a coarse bound on the angle; finding the exact average length gives more precision (Theorem \ref{thm:irr}). \label{table:angles}}
  \end{center}
\end{table}

\begin{question}
Given the set of all words in $\L_P$ of length at most $N$, how precise of an approximation of $P$ can one construct?
\end{question}

This observation about $\L_P$ and coarse bounds on angles implies that it is sometimes possible for a single bounce word to distinguish between tables. For example, a rhombus with angle $\pi/3$ between $A$ and $B$ admits the bounce word $ABA$, by starting on $A$ and shooting perpendicular to $B$ (Figure \ref{fig:words}(a)). However, a square billiard table with edges consecutively labeled $A,B,C,D$ does not have any trajectory containing bounce word $ABA$; a trajectory hitting $A$ and then $B$ must hit $C$ or $D$ next (Figure \ref{fig:words}(b)).

Finally, our methods for determining adjacency and angles use only the local structure of the table around each vertex and developments along billiard trajectories. As a consequence, our methods apply to polygons that are not simply connected, polygons that wrap around themselves, and polygons with angles greater than $2\pi$ (see Figure \ref{fig:not-sc}). The exact class of these ``generalized polygons'' is discussed in \S\ref{sec:background}. Note that our methods cannot determine other aspects of such polygons, such as the location of the ``hole'' in the left polygon of Figure \ref{fig:not-sc}.\\

\begin{figure}[h]
  \centering
  \includegraphics[width=3in]{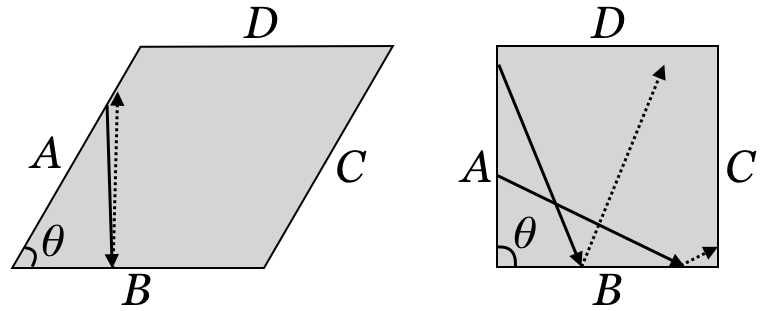}
  \caption{A rhombus with angles $\pi/3$, $2\pi/3$ contains the bounce word $ABA$, but a square table does not; this is an example of the calculations in Table \ref{table:angles}.}
  \label{fig:words}
\end{figure}

\subsection{Selected results on spectral rigidity}\label{subsec:previous}

In Kac's original question, the ``drums'' were connected planar domains and their ``sound'' was the spectrum of the corresponding Laplacian (note that in this setting, eigenvalues of the Laplacian may be identified with overtones, so the question is not far divorced from physical reality). Using a method of Sunada \cite{Sunada}, Gordon, Webb, and Wolpert found two--dimensional counterexamples to Kac's question, i.e., domains whose Laplacians have the same spectrum \cite{GWWdrum}. However, Zelditch has proven that if the domains are convex and have certain symmetry and regularity properties, then the spectrum of the Laplacian does distinguish these domains \cite{Zelditch}.

\begin{figure}[t]
\centering
\includegraphics[width=5in]{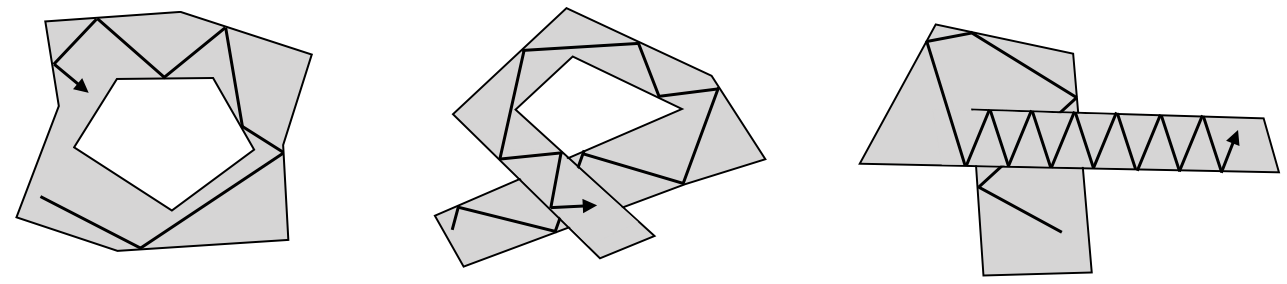}
\caption{Generalized polygonal billiard tables, with an example trajectory on each. Our methods for finding adjacency and angles also apply to such polygons.}
\label{fig:not-sc}
\end{figure}

Now let $M$ be a surface equipped with a hyperbolic metric $\varphi$ of constant curvature. By the Selberg trace formula, the spectrum of its Laplacian determines its {\em marked length spectrum}, the function that assigns to every free homotopy class of loops in $M$ the $\varphi$--infimal length of a representative. Even when $\varphi$ is not hyperbolic, one may still ask if the $\varphi$--marked length spectrum determines the metric.

It is a classical result of Teichm{\"u}ller theory that the lengths of only $9g-9$ simple closed curves are needed to distinguish hyperbolic structures of constant curvature $-1$ on a closed surface of genus $g \ge 2$, and Hamenst{\"a}dt has shown that $6g-5$ curves (and no fewer) suffice \cite{6g-5}.

Otal proved that the entire marked length spectrum on surfaces with Riemannian metrics of negative curvature distinguishes these metrics \cite{Otal}.
This result was subsequently generalized to certain metrics of nonpositive curvature by Croke, Fathi, and Feldman in various combinations (see \cite{Croke}, \cite{Fathi}, \cite{CFF}).

In the non-Riemannian setting, Bonahon proved that the marked length spectrum can no longer differentiate between all metrics on a surface \cite{BonahonCounter}. All the same, Duchin, Leininger, and Rafi proved that the marked length spectrum distinguishes between flat cone metrics coming from quadratic differentials [\cite{DLR}, Theorem 1]. 
Bankovic and Leininger extended this result to all nonpositively curved flat cone metrics \cite{BL}.

By unfolding a polygonal table to a flat cone surface and using the symbolic coding of geodesics coming from the table, the authors of \cite{DELS} are able to use the results of \cite{BL} in their proof of Theorem \ref{thm:DELS}. The rigidity of the bounce spectrum can therefore be seen as a natural combinatorial extension of the inverse spectral problem.

Duchin, Leininger, and Rafi also showed no finite set of curves in the marked length spectrum distinguishes the flat cone metrics associated to a quadratic differential [\cite{DLR}, Theorem 3]. This result can be compared to our Theorem \ref{thm:no_finite}.

The above results show that marked length spectra are in some contexts complete invariants of a metric on a surface---they abstractly determine the metric that induces them. In these cases one can take up the corresponding reconstruction questions, as we do in this paper for the bounce spectrum. For example, it is immediate that one can reconstruct a polygonal presentation for a quadratic differential given the holonomy of every saddle connection. However, it is unknown if knowing only the lengths of the saddle connections is enough to derive the same result.

\begin{question}
Given the marked length spectrum of a flat cone metric coming from a quadratic differential $q$ on a closed surface of genus at least 2, can one reconstruct a polygonal presentation for $q$?
\end{question}

\subsection{Selected results on symbolic dynamics}

Our work on the bounce spectrum is related to a large amount of literature on bounce sequences, and the related \emph{cutting sequences} on translation surfaces.

While our goal is to start with the bounce spectrum and reconstruct the billiard table, most of the literature on bounce sequences and cutting sequences begins with the table or surface, and describes its spectrum of sequences. Morse and Hedlund \cite{MH} worked on classifying cutting sequences on \emph{Veech surfaces} nearly a century ago, and more recently Smillie and Ulcigrai \cite{SU1, SU2}, Davis \cite{Davis1, Davis2} and Davis, Pasquinelli and Ulcigrai \cite{DPU} classified cutting sequences on specific cases of Veech surfaces.

Certainly, the work of our colleagues \cite{DELS} uses different methods to address the same questions that we consider here. Prior work of Bobok and Troubetzkoy \cite{CodeOrder} also proved a similar result to our Theorem \ref{thm:hearing_angles}, under the assumption that the billiard table was rational and that there exists a point whose return map to the boundary of the polygon is minimal. In fact, it is an interesting coincidence of convergent mathematics that our Theorem \ref{thm:hearing_angles} and their Theorem 7.1 result in similar pictures, despite using different technical machinery.

Bobok and Troubetzkoy have also proven that the set of periodic orbits on a rational table, i.e., the periodic bounce spectrum, is enough to determine a (non-right-angled) rational table \cite{BobTroub}. This result can also be re-derived using \cite[\S 5.2]{DELS}. There should be a method, therefore, to reconstruct rational tables using only this information.

\begin{question}\label{question:rational_from_periodic}
  Can one reconstruct a rational billiard table from its periodic bounce spectrum?
\end{question}

\subsection{Reconstructing edge lengths for a right-angled table is impossible}\label{subsec:rect}

We conclude the introduction with a proof that it is impossible to use the bounce spectrum to differentiate between right-angled tables that are related by an affine transformation.

\begin{definition}
  A billiard table is \emph{right-angled} if all of its angles are $\pi/2$ or $3\pi/2$.
\end{definition}

\begin{proposition}\label{prop:rectilinear}
  Two right-angled billiard tables that are related under edge-parallel stretching, i.e. under an affine transformation of the form $$\left[\begin{matrix}a&0\\ 0&b\end{matrix}\right]$$ for tables with horizontal and vertical edges, have the same bounce spectrum.
\end{proposition}

\begin{proof}
A stretch (expansion or contraction) of a billiard table parallel to one of its edges changes the angles of a trajectory bouncing off that edge, but preserves angle equality. An expansion of a billiard table in one direction is equivalent to a contraction of the table in the perpendicular direction, by scaling the picture. Each edge of a right-angled table is parallel or perpendicular to every other edge, so a stretch parallel to any edge preserves angle equality for bounces off of any other edge of the table (Figure \ref{fig:rectilinear}). Since each billiard trajectory on a right-angled table is also a billiard trajectory on the tables that are its affine image under a horizontal or vertical stretch, such stretches preserve the bounce spectrum.
\end{proof}

\begin{figure}[!h]
  \centering
  \includegraphics[width=4in]{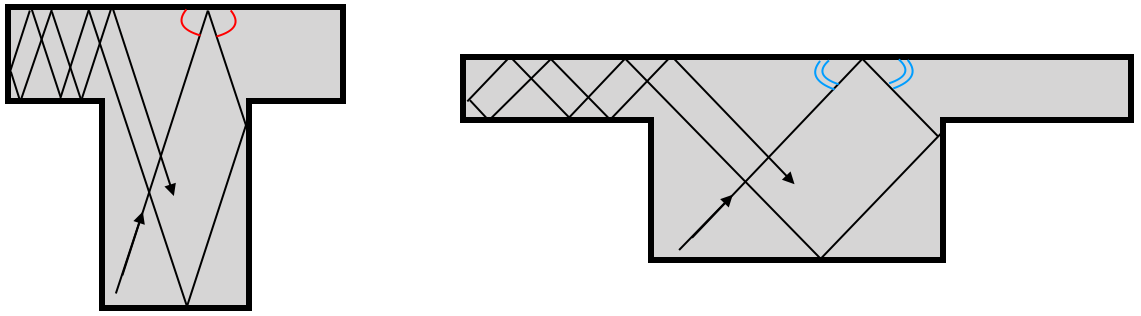}
  \caption{Two right-angled tables that are related by a vertical compression and horizontal elongation, with a billiard trajectory. These transformations preserve angle equality for bounces against horizontal and vertical edges and thus preserve billiard trajectories on right-angled tables.}
  \label{fig:rectilinear}
\end{figure}

\begin{corollary}
  The shape of a right-angled table $P$ can at best be reconstructed from $\B(P)$ up to equivalence under edge-parallel stretching.
\end{corollary}

\subsection{Outline of the paper}

We begin in \S\ref{sec:background} by defining the bounce spectrum and fixing our notation for the rest of the paper.
We also make precise what we mean by ``reconstructing the table from its bounce spectrum.'' We also collect results about how trajectories produce bounce sequences and how in turn a bounce sequence may be realized by trajectories.

We investigate how the geometry of a polygon encodes itself in its bounce spectrum in \S\S \ref{sec:emitters} --\ref{sec:angles}; this forms the technical heart of the paper. In \S\ref{sec:emitters}, we define {\em common prefixes} and {\em ideal trajectories}, concepts that allow us to work with singular trajectories in a coherent way. We then make use of these notions extensively in \S\ref{sec:adjacency} and \S\ref{sec:angles}, in which we prove our main theorems on reconstructing a billiard table from its bounce spectrum.

To construct adjacency and angles in sections \S\ref{sec:adjacency} and \S\ref{sec:angles}, we crucially use infinite sequences. In \S\ref{sec:no_finite}, we show that this is the best we can possibly do: it is not possible for a finite amount of information from the bounce spectrum to determine the polygon (Theorem \ref{thm:no_finite}).

\subsection{Acknowledgments} We thank Moon Duchin for organizing the Polygonal Billiards Research Cluster, for suggesting this problem, and for providing leadership and motivation throughout the project. We also thank the other participants and visitors of the cluster for many interesting and productive discussions about this problem, especially the authors of \cite{DELS}, Curt McMullen, Ronen Mukamel, and Rich Schwartz. We are grateful to Dan Margalit and Serge Troubetzkoy for helpful comments.\\

\noindent \textbf{Funding.} This work was initiated in the Polygonal Billiards Research Cluster held at Tufts University in Summer 2017 and was supported by the National Science Foundation under grant [DMS-CAREER-1255442]. AC and JL were also partially supported by NSF grants [DGE-1122492] and [DGE-1650044], respectively.

\section{Background}\label{sec:background}

\subsection{Definition of the bounce spectrum}

As noted in the introduction, our results will be applicable to a wider class of billiard tables than standard polygons. To describe the class precisely, let $X$ be a simplicial 2-complex. A {\em Euclidean realization} of $X$ is an assignment of lengths to each edge such that the edge lengths of each simplex satisfy a non-degenerate triangle inequality. Such a realization allows us to identify each 2-simplex with a (non-degenerate) Euclidean triangle.

\begin{definition}
  A {\em generalized polygon} is the piecewise Euclidean complex obtained as a Euclidean realization of a
  pure,
  finite, connected, simplicial, 2-complex $X$
  such that the link of every vertex is a path.
\end{definition}

With this definition, a polygon is simply a generalized polygon whose interior isometrically embeds into the plane.

If $P$ is a generalized polygon, we set $\partial P$ to be the set of edges of $P$ which are contained in only one face. A {\em marked generalized polygonal billiard table} is a generalized polygon $P$ together with a labeling of its boundary edges where the labels come from some finite
alphabet $\A$. The billiard flow $\phi_t$ on $P \times S^1$ is given as a piecewise linear flow with optical reflection at edges of $\partial P$. As for standard polygonal tables, we fix the convention that the flow is stationary at corners.

In the sequel, for ease of exposition we usually restrict our discussion to standard polygons. That said, all of our results
(and the relevant results from the literature, see forthcoming work of Yunzhe Li \cite{Troubetzkoy})
%\cite{nonSCGKT})
hold in the generalized setting.

For a pair $(p,\theta) \in P\times S^1$ (where we identify $S^1$ with $[0, 2\pi)$), we define the trajectory $\tau (p,\theta)$ to be the orbit of $p$ under both the forward and backwards billiard flow.
 We can also define the forward trajectory $\tau_+(p,\theta)$ by taking only the forward billiard flow.

 If a trajectory $\tau(p,\theta)$ ever hits a vertex of $P$, then we say that it is \emph{singular}. If $\tau(p,\theta)$ is nonsingular, then the \emph{(full) bounce sequence}, $\B(p,\theta) \in \A^\Z$ is the bi-infinite indexed string of symbols in the alphabet $\A$ encoding the edges traversed by $\tau (p,\theta)$. The \emph{forward bounce sequence} $\B_+(p,\theta) \in \A^\N$ is the sequence of edges traversed by a nonsingular forward trajectory $\tau_+(p,\theta)$.
 We fix the convention that if $p \in \partial P$, say on edge $A$, then $\B_+(p,\theta)$ does not start with edge $A$, but rather the next edge that $\tau_+(p,\theta)$ meets.

\begin{definition}
The \emph{bounce spectrum} of a (generalized) polygon $P$ is the collection $\B(P)=\{\B(p,\theta)\}$ of the bounce sequences of all nonsingular trajectories on $P$.
The \emph{bounce language} $\L=\L_P\subset \A^*$
is the language of finite subwords of $\B(P)$. Similarly, we define the \emph{forward bounce spectrum} $\B_+(P)$ as the collection of all $\B_+(p,\theta)$ such that $\tau_+(p,\theta)$ is nonsingular.
\end{definition}

Note that $\tau_+(p, \theta) = \tau_-(p, -\theta)$, so $\B_+(P) = \B_-(P)$, so the latter notation for backward bounce spectrum is unnecessary:
\begin{observation}
For any (generalized) polygon $P$, $\B_+(P) = \B_-(P)$.
\end{observation}

\begin{remark}\label{rem:singular-word}
Observe the subtle differences between between $\L$ and the set of all words that can be realized by any trajectory, and between $\B_+(P)$ and infinite tails of elements of $\B(P)$. For the former, singular trajectories may realize words before reaching a vertex, and for the latter we may have $(p,\theta)$ with a singular backwards trajectory and non-singular forward trajectory.

In \S \ref{realizing} we resolve the former issue by showing that any words realized by singular trajectories are also realized by non-singular trajectories (Corollary \ref{coro:language_clarification}). Similarly, the latter issue is resolved in \S \ref{sec:full_determines_forward}, where we characterize exactly which bounce sequences are realizable (Theorem \ref{thm:detecting_forward_realizability}).
\end{remark}

\subsection{Developments}

\indent One key tool in analyzing trajectories and bounce sequences is to \emph{unfold} copies of a marked billiard table along a trajectory. Take a sequence of edges $(E_i)$, which may be either finite or infinite. 
The \emph{unfolding} or \emph{development} $\D_P((E_i))$ of a polygon $P$ along a given sequence of edges $(E_i)$ is the polygonal complex whose faces are
\[(P, r_1 P, r_2 r_1 P, r_3 r_2 r_1 P, \dots ),\]
where $P$ stays in place, $r_i P$ is the reflection of $P$ over the edge labeled by $E_i$, and successive polygons are identified along their reflecting edge: $r_j \dots r_1 P$ is glued to $r_{j+1} r_j \dots r_1 P$
along their edges labeled $E_{j+1}$. Here each $E_i$ refers to edges of different reflected copies of $P$, so the edge labels $E_1, \ldots, E_{j+1}$ are not necessarily distinct.

Observe that if $P \subset \C$ is a planar polygon, then there is a natural projection of $\D_P(w)$ to $\C$.

The development of a polygon inherits a natural piecewise Euclidean metric from $P$, and by extending a choice of positive $y$-direction from our original $P$ the development may also be equipped with a consistent choice of positive $y$-direction.

Endow $P$ with an orientation, say counterclockwise. Then a reflection over an edge of $P$ is orientation reversing. Moreover, the orientation of an image of $P$ under some number of reflections, say $n$, is counterclockwise if and only if $n$ is even.
In particular, we see that the development $\D_P(w)$ is tiled by copies of $P$ with alternating orientation.

If $w=E_1 \ldots E_n$ is a word in $\L$, then we will abuse notation and set
\[\D_P(w) := \D_P((E_i)_{i=1}^n).\]
We will also often write $E_j \dots E_1 P$ to mean $r_j \dots r_1 P$ when it is clear from context.

\subsection{Topology of the bounce spectrum}\label{sec:spectrum}

We may topologize $\B(P)$ by viewing it as a subspace of the sequence space $\A^\Z$. 
The topology on $\A^\Z$ is the standard topology generated by a basis of cylinder sets, with the discrete topology on $\A$.

\begin{definition}\label{def:cylinder2ind}
Given a polygon $P$ with edges $\set{E_1, E_2, \ldots, E_n}$, a \emph{cylinder set} of $\A^\Z$ is any set of the form \[ \mathcal{C}_{[i_1,i_2]}(E_1, E_2, \ldots, E_N) = \cdots \times \A \times \A \times E_1 \times E_2 \times \cdots \times E_N \times \A \times \A \times \cdots, \] where the fixed letters begin at index $i_1$ and end at index $i_2 = i_1 + N$, for any indices $i_1, i_2$, any finite length $N \in \N$, and any edges $E_1, E_2, \ldots, E_N$.
For brevity, will often suppress the indices when convenient.
\end{definition}

We denote the closure of $\B(P)$ in $\mathcal{A}^\Z$ by $\overline{\B(P)}$. Just as with $\B(P)$, it will at times be beneficial to consider the closure of $\B_+(P)$ in $\mathcal{A}^\N$, denoted $\overline{\B_+(P)}$. Observe that by Tychonoff's theorem, both of these sets are compact.

\begin{observation}\label{obs:full_bounce_not_dense}
$\B(P)$ is not dense in $\A^\Z$.
\end{observation}
\begin{proof}
Let $E$ be an edge of $P$ and consider the cylinder $\mathcal{C}(E,E)$. This is an open set in $\A^\Z$, and its intersection with $\B(P)$ is empty, because a billiard trajectory can never bounce off of the same edge twice in a row.
\end{proof}

Therefore $\B(P)$ is ``not too big'' in $\A^\Z$. However, $\B(P)$ is still very large, because every pair of trajectories that are not parallel must have bounce sequences that eventually disagree:

\begin{lemma}\label{lem:uncountable}
$\B(P)$ is uncountable.
\end{lemma}

\begin{proof}
Let $p$ be a point in the interior of $P$. There are a countable number of singular directions from $p$, so there are uncountably many nonsingular directions. Let $\tau_1$, $\tau_2$ be different nonsingular trajectories from $p$.

Unfold copies of $P$ along $\tau_1$ and $\tau_2$. Eventually, some vertex of $P$ must occur in the unfolding between $\tau_1$ and $\tau_2$, because the distance between trajectories eventually exceeds the diameter of $P$. After a vertex comes between them, the cutting sequences corresponding to $\tau_1$ and $\tau_2$ are different.

Since there are uncountably many such nonsingular trajectories, there are uncountably many corresponding distinct cutting sequences.
\end{proof}

\subsection{Realization of bounce sequences}\label{realizing}

Given a word  $E_1 \ldots E_n \in \A^*$, we say that $\tau(p,\theta)$
\emph{realizes} the word if the first $n$ edges traversed by $\tau(p,\theta)$ are $E_1, \ldots, E_n$.
In particular, a trajectory can realize a word even if the trajectory is singular.

While every trajectory determines a unique bounce sequence, a bounce sequence may be realized by many trajectories, or by none at all. To investigate these possibilities, we define below a geometric interpretation of a cylinder set as a region within the development of a polygon.

\begin{definition}\label{def:corridor}
Let $(E_i)$ be a sequence of edges. The {\em corridor} corresponding to $(E_i)$ is the set of points in $\D_P({(E_i)})$ that lie on a trajectory realizing $(E_i)$.
\end{definition}

\begin{remark}
A nonsingular trajectory $\tau(p, \theta)$ lies in the corridor corresponding to $E_1 \ldots E_N$ if and only if $\B(p, \theta)$ lies in the cylinder set $\mathcal{C}(E_1, \ldots, E_N)$.
\end{remark}

A corridor is finite in length if $(E_i)$ is finite; in that case, it may have ``flared'' ends since the finite trajectories need not be parallel (Figure \ref{fig:flared-corridor}).
We define the {\em width} of a finite corridor to be the infimal $\varepsilon$ so that for every trajectory $\tau$ lying in the corridor, the corridor is entirely contained in the $\varepsilon/2$--tubular neighborhood of $\tau$.

A corridor could be infinite in length if, for instance, $(E_i)$ is periodic. In this case, the corridor is the geometric locus of the maximal family of parallel trajectories realizing the word. In this case we may therefore measure the \emph{width} of an infinite corridor perpendicular to its defining family of parallel trajectories. This definition coincides with the one given above when the bounce sequence is periodic.

\begin{figure}[h]
\begin{subfigure}[b]{.46\textwidth}
\centering
\includegraphics[width=.95\textwidth]{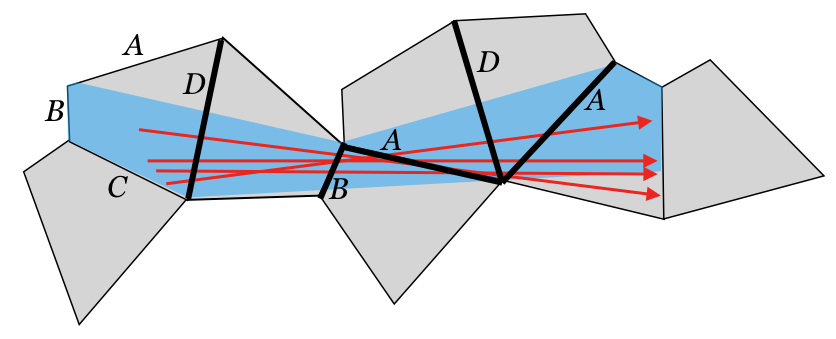}
\caption{A corridor for the finite word $DBADA$, with several example trajectories that realize this word.}
\label{fig:flared-corridor}
\end{subfigure} \ \
\begin{subfigure}[b]{.52\textwidth}
\centering
\includegraphics[width=.95\textwidth]{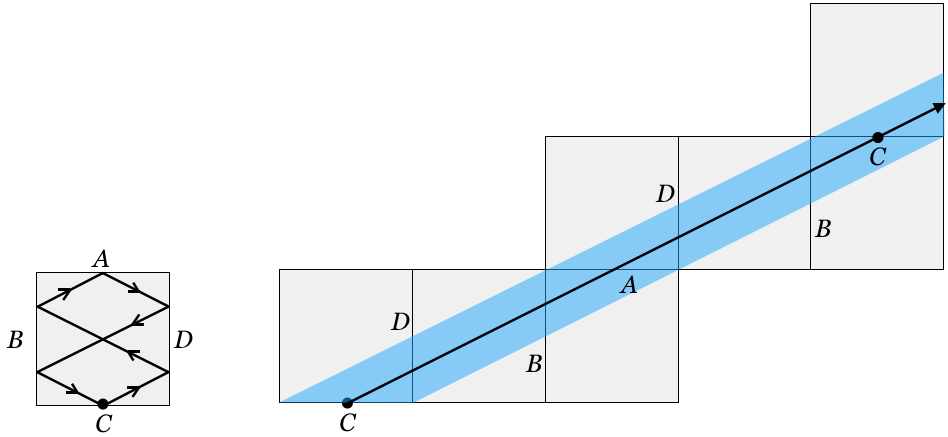}
\caption{Part of an infinite corridor realizing the periodic word $CDBADB$ in the square.}
\label{fig:parallel-corridor}
\end{subfigure}
\caption{Finite and infinite corridors}
\end{figure}

\begin{lemma}\label
{lem:neighborhood_in_corr}
Let $\tau$ be a nonsingular trajectory lying in a corridor about a (finite) word $w=E_1 \ldots E_n$. Then for some $\varepsilon >0$, the $\varepsilon$-tubular neighborhood about $\tau$ also lies in the corridor.
\end{lemma}

Observe that this lemma also implies that the corridor about any finite word also has positive width.

\begin{proof}
Take the development $\D_P(w)$ of $P$ over $w$.
Let $\tau=\tau(p,v)$ be a trajectory realizing $w$, possibly singular outside of $\D_P(w)$.

Then $\tau$ intersects edges $E_1$ through $E_n$ in their interiors.
Hence there exists an $\varepsilon$ neighborhood of $\tau \cap \D_P(w)$ that does not contain any vertices. Every trajectory in this neighborhood parallel to $\tau$ also realizes $w$, and hence all lie inside of the corridor.
\end{proof}

In particular, since there are uncountably many trajectories running through this tubular neighborhood, and only countably many of these may be singular, this tells us that every word realized by a trajectory is realized by a \emph{nonsingular} trajectory. This resolves the concern discussed in Remark \ref{rem:singular-word} about the possibility of words arising only from singular trajectories:

\begin{corollary}\label{coro:language_clarification}
A word $w \in \A^*$ is realized by some $\tau(p,\theta)$ if and only if $w \in \L$.
\end{corollary}

While $\tau$ has a neighborhood of parallel trajectories about it that all realize the word $E_1 \dots E_n$, note that by definition, \emph{any} trajectory that fits within a corridor also realizes the word $E_1 \dots E_n$, even if it is not parallel to $\tau$. However, we can give a bound on how ``far apart" two trajectories may be while still realizing the same word. We will use the following lemma throughout the paper to show that two trajectories that realize the same word have to be close to parallel.

We first set notation for expressing the distance between the points where a trajectory intersects two edges. If $(p,\theta) \in P \times [0, 2 \pi)$ and $(E_i) = \B_+(p,\theta)$, then we denote the \emph{translation distance} between $\tau(p,\theta) \cap E_i$ and $\tau(p,\theta) \cap E_j$ by $d_{\tau(p,\theta)}(E_i E_{i+1} \ldots E_j)$. Note that the reason we need to specify a basepoint $(p,\theta)$ for the trajectory is because the trajectory $\tau$ may realize the word $E_i \ldots E_j$ multiple times without being periodic. Two different occurrences of a word in an infinite sequence may correspond to different translation distances.

\begin{lemma}\label{lem:realizability}
Let $w = E_1 \ldots E_n$ be any word in $\L$, and let $(p, \theta) \in P \times [0, 2\pi)$ be such that $\tau_+(p, \theta)$ realizes $w$. Then if $(p', \theta') \in P \times [0, 2\pi)$ is such that $\tau_+(p', \theta')$ realizes $w$,
\begin{equation}\label{angle-condition}
| \theta - \theta' | \le
\tan^{-1}\left( \frac{2 \cdot \textrm{diam}(P)}{d(E_1 \dots E_n)}\right).
\end{equation}
where
${d(E_1 \dots E_n)} = \inf_{(p, \theta)} {d_{\tau(p,\theta)}(E_1 \dots E_n)}$
is the infimum translation distance over all trajectories that realize $w$.
\end{lemma}
\begin{proof}
Without loss of generality assume that $\theta=0$, and take the development $\D_P(w)$. Consider the segment $L$ of $\tau(p,\theta)$ connecting $\tau_+(p',\theta') \cap E_1$ and $\tau_+(p',\theta') \cap E_n$.

As $(p',\theta')$ realizes $w$, $L$ lies in the corridor corresponding to $w$. The horizontal translation of $L$ is then at least $d(E_1, \ldots, E_n)$, and its vertical translation is at most the width of the corridor, which is at most $2 \cdot \text{diam}(P)$, since the development itself has width at most $2 \cdot \text{diam}(P)$.
\end{proof}

\begin{corollary}\label{cor:realizability}
If $\tau(p_{(N)}, \theta_{(N)})$ and $\tau(p'_{(N)}, \theta'_{(N)})$ are sequences of trajectories that realize $E_1 \ldots E_N \in \L_P$, then $\theta'_{(N)}$ and $\theta_{(N)}$ must converge to the same limit as $N \rightarrow \infty$.
\end{corollary}

\subsection{Aperiodicity and uniqueness of realizations}

It is a standard result that a trajectory is periodic if and only if it has a periodic bounce sequence, hence any aperiodic trajectory has an aperiodic bounce sequence. We show in this section that, up to choosing basepoints, an aperiodic realizable bounce sequence has a \emph{unique} realizing trajectory.

\begin{theorem}[Galperin--Kruger--Troubetzkoy, \cite{GKT}, Theorem 2]\label{thm:GKT}
For any polygonal table $P$ and aperiodic $(E_i) \in \A^\N$, there exists at most one pair $(p, \theta) \in \partial P \times S^1$ such that $\B_+(p,\theta) = (E_i).$
\end{theorem}

\begin{corollary}\label{coro:aperiodic_traj_are_unique}
On any polygonal table, there exists at most one trajectory realizing a given aperiodic bounce sequence $(E_i)_{i \in \Z}$.
\end{corollary}
\begin{proof}
Apply Theorem \ref{thm:GKT} twice to get two pairs $(p_\pm, \theta_\pm)$ of points and angles such that $(p_+, \theta_+)$ realizes $(E_i)_{i=1}^\infty$ and $(p_-, \theta_-)$ realizes $(E_i)_{i=-\infty}^{-1}$. In order for a trajectory $\tau$ to realize $(E_i)_{i \in \Z}$, it must be that $p_+=p_-$ lies on the edge labeled by $E_0$ and $\theta_+$ and $\theta_-$ obey the law of optical reflection. These points and angles are unique, so there is at most one trajectory realizing $(E_i)_{i \in \Z}$.
\end{proof}

The uniqueness of realizations of realizable aperiodic bounce sequences implies the following useful result, which tells us that as we develop along any aperiodic trajectory, the width of the corridor associated to the corresponding word goes to $0$. We use the following Corollary in \S \ref{sec:adjacency}, when we show that we can use the bounce spectrum to construct adjacency.

\begin{corollary}\label{coro:thinning_corridor}
Let $\tau$ be an aperiodic trajectory on a table $P$. Then $\tau$ passes arbitrarily close to the vertices of $P$.
\end{corollary}
\begin{proof}
Suppose there exists some $\varepsilon > 0$ such that the $\varepsilon$ neighborhood about $\tau$ contains no vertices of $P$. Then the neighborhood contains a family of parallel trajectories, contradicting the uniqueness of the realization of the bounce sequence of $\tau$.
\end{proof}

As mentioned earlier, given a polygonal table $P$ with edges labeled in an alphabet $\A$, one can consider the collection $\A^\Z$.
We will often find it useful to think of $\A^\Z$ as a topological space and to consider the topological closure of $\B(P)$ within it. The goal of this section is to describe the topologies on $\A^\Z$ and $\B(P)$.

\begin{remark}\label{rem:repeat}
It is clear that the bounce spectrum $\B(P)$ is a proper subset of $\A^\Z$, since for instance a polygonal table does not admit a bounce sequence where an edge label occurs twice in a row.
\end{remark}

\begin{question}
Other than those with consecutive repeated edge labels, are there any other words that never appear in the bounce language of \emph{any} billiard table?
\end{question}

\section{Common prefixes and ideal trajectories}\label{sec:emitters}

\subsection{Common prefixes}

Our next goal is to identify how the adjacency structure of $P$ encodes itself in the bounce spectrum.
If two edges $A$ and $B$ meet in a vertex, then there are points on $A$ that are arbitrarily close to $B$, and vice versa. Now under the billiard flow, if $(p,\theta)$ and $(p',\theta')$ are close to each other, then they stay close to each other for a definite amount of time. This means that if $A$ and $B$ are adjacent, we can find trajectories emanating from points on $A$ that can track trajectories emanating from points on $B$ for an arbitrarily long time.

This fact can be visualized as in Figure \ref{fig:emitters}. Each of the edges of our polygon $P$ may be thought of as a neon sign of a different color radiating light in every direction.

\begin{figure}[!h]
\centering
\includegraphics[width=3.5in]{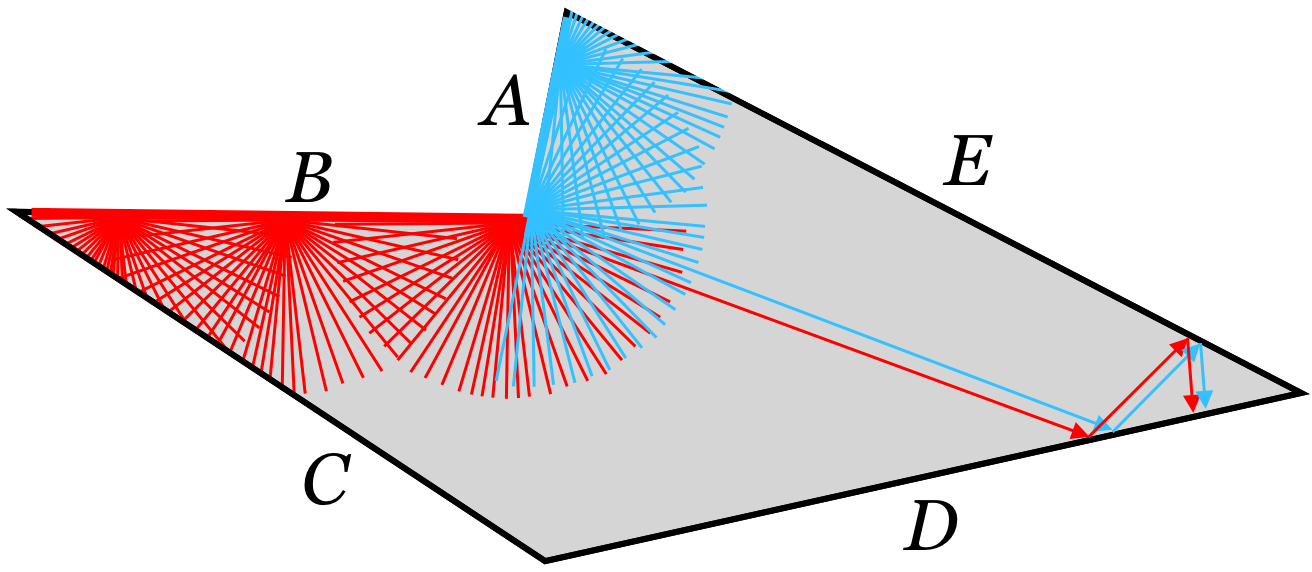}
\caption{The blue and red edges $A$ and $B$ form a set of common prefixes for the sequence $DED\ldots$.}
\label{fig:emitters}
\end{figure}

We now rephrase this visualization in terms of bounce sequences:

\begin{definition}
We say that a set of edges $\{A_1, \ldots, A_k\}$ is a set of \emph{common prefixes} if there exists a sequence $(E_i) \in \A^\N$ such that
\begin{center}
  {\em $(A_j, E_1, E_2, \ldots) \in \overline{\B_+(P)}$ for each $j=1, \ldots, k$.}
\end{center}
In this case we say that $A_j$ is a \emph{prefix} for the sequence $(E_i)_{i \in \N}$.
We may also sometimes speak of an edge being a prefix for a trajectory whenever this trajectory realizes a bounce sequence $(E_i)$ that has the edge as a prefix.
\end{definition}

Note that in section \S\ref{subsec:general_angles}, we will use \emph{insertions}, which are letters (edge labels) inserted into the ``middle'' of a bi-infinite sequence, which is an extension of the idea of the prefix that we use here for a one-sided infinite sequence.

\begin{definition}
  We say that a set $\{A_1, \ldots, A_n\}$ is a \emph{realizable set of common prefixes} if there exists an $(E_i) \in \B_+(P)$ such that $\{A_1, \ldots, A_n\}$ are all prefixes for $(E_i)$.
\end{definition}

Not all sets of common prefixes are realizable.
For example, consider the non-convex hexagon in Figure \ref{fig:not_realizable}.
Choose some point $a$ on $A$ and a sequence of angles $\theta_n$ approaching $2\pi$ (measured counter-clockwise from the horizontal) such that $\tau_+(a, \theta_n)$ are all nonsingular.
The limit of $\B_+(a,\theta_n)$ will then define some point
$(A, E_1, E_2, \ldots) \in \overline{\B_+(P)}$.
By choosing an appropriate sequence of directions $\eta_n$ we may approximate $(E_i)_{i=1}^\infty$
by trajectories starting from a point $b$ on $B$ and conclude similarly. Thus $\{A, B\}$ is a set of common prefixes. However, the only sequences approximated by $A$ and $B$ all
define trajectories that limit to the line containing $A$ and $B$
(see Lemma \ref{lem:collinear_limiting_points}), and so $\{A,B\}$ is not realizable.

\begin{figure}[h]
  \centering
  \includegraphics[width=5.5in]{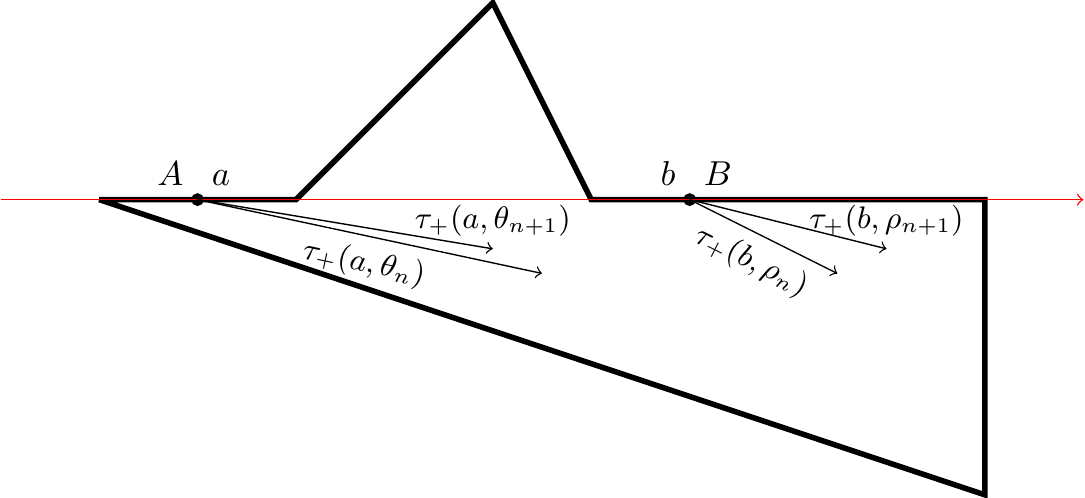}
  \caption{A set of common prefixes $\{A,B\}$ that is not realizable.}
  \label{fig:not_realizable}
\end{figure}

Observe that we cannot currently say anything about the realizability of a set of common prefixes using only information in $\B(P)$, since it is defined in terms of sequences appearing in $\B_+(P)$.
We will see in \S \ref{sec:full_determines_forward} that $\B(P)$ determines $\B_+(P)$. Before attending to this, we will first investigate bounce sequences that are not realizable.

\subsection{Ideal trajectories}\label{sec:ideal}

The following section is written in terms of bi-infinite bounce sequences, but similar results hold for forward bounce sequences, and the proofs are essentially identical to those presented below.

To consider non-realizable sequences as geometric objects, we use our understanding of the topology of $\overline{B(P)}$; in particular, we have that
every sequence $(E_i)_{i \in \Z} \in \overline{B(P)}$ can be expressed as the limit of bounce sequences in $\B(P)$.

\begin{definition}
  Given a sequence $(E_i)_{i \in \Z} \in \overline{\B(P)}$, we define an \emph{ideal trajectory} to be an infinite line in $\D_P((E_i)_{i \in \Z})$, possibly containing vertices. \end{definition}

  Using this language we say the ideal trajectory is associated to $(E_i)_{i \in \Z}$ and vice-versa. As a geometric object, an ideal trajectory should be thought of as a continuation of a (singular) trajectory along a development. Ideal trajectories can be constructed as follows: take $(p_n,\theta_n) \in P \times [0, 2\pi)$ such that $\tau(p_n, \theta_n)$ is nonsingular and the associated bounce sequences $\B(p_n, \theta_n)$ converge to $(E_i)$ in $\overline{\B(P)}$. This can be seen in Figure \ref{fig:ideal_trajectory} where the trajectories $\tau_n$ limit to the ideal trajectory $\tau$; note that in this example $p_n = p$ for all $n$. More formally, by Arzela-Ascoli, the trajectories $\tau(p_n, \theta_n)$ limit to some line lying inside of the development.

\begin{figure}[h]
\centering
\includegraphics[width=6in]{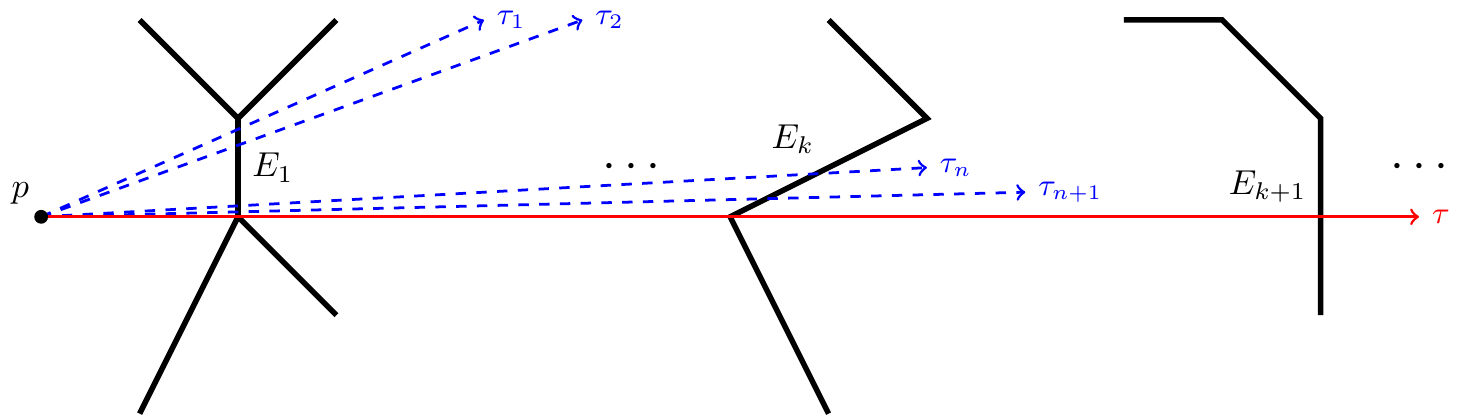}
\caption{An ideal trajectory (solid red) as the limit of nonsingular trajectories (dashed blue) through polygon edges (thick black).}
\label{fig:ideal_trajectory}
\end{figure}

Note that nonsingular trajectories are also ideal trajectories; one can trivially take the constant sequence of bounce sequences for the geometric construction, or consider the trajectory itself as the line in the development. Because of this, we focus our attention on ideal trajectories.

In Corollary \ref{coro:aperiodic_traj_are_unique}, we saw that an aperiodic bounce sequence has (at most) one trajectory that realizes it. We have a similar result for any aperiodic sequence in $\overline{\B(P)}$.
\begin{lemma}\label{lem:ideal_uniqueness}
There is exactly one ideal trajectory that realizes any aperiodic sequence $(E_i)_{i \in \Z} \in \overline{\B(P)}$.
\end{lemma}
\begin{proof}
If $(E_i)_{i \in \Z} \in \B(P)$ this follows by Corollary \ref{coro:aperiodic_traj_are_unique}. If instead $(E_i)_{i \in \Z} \not\in \B(P)$, we know that since $(E_i)_{i \in \Z} \in \overline{\B(P)}$ there exists some sequence of trajectories  $\tau(p_n, \theta_n)$ whose bounce sequences limit to $(E_i)_{i \in \Z}$. By Arzela-Ascoli, these trajectories limit to some line lying inside of the development. We now argue that this limit line is unique. Any other sequence of trajectories whose bounce sequences limit to $(E_i)_{i \in \Z}$ also has a limit line. By Corollary \ref{cor:realizability}, the angles between the trajectories in the two sequences goes to 0 as $n$ goes to infinity, and so the two limit lines must be parallel.

In fact, the limiting lines must coincide. Otherwise, they bound a flat strip in the development whose interior contains no vertices. Any trajectory running through the strip is hence nonsingular and must realize $(E_i)_{i \in \Z}$. But we assumed that $(E_i)_{i \in \Z}$ was not realizable, a contradiction, hence the limiting ideal trajectory must be unique.
\end{proof}

\section{Reconstructing adjacency of edges}\label{sec:adjacency}

We now use the technology of common prefixes to reconstruct adjacency of sides from a polygon's bounce spectrum. As a part of proving this result, we show that the bounce spectrum of a billiard table determines its forward bounce spectrum. This resolves the concern about the potential discrepancy between the information contained in two sets discussed in Remark \ref{rem:singular-word}.

The idea behind our analysis of adjacency is that if two edges $A,B$  are adjacent, then there is a vertex between them and there is a trajectory starting at this vertex that is nonsingular in the forward direction (Figure \ref{fig:adjacency-idea}). Thus, we can find trajectories close to the vertex, one starting along $A$ and one starting along $B$, whose bounce sequences match that of the singular trajectory for arbitrarily many bounces. The following sections make this idea precise and complete, using the ideas of common prefixes and ideal trajectories developed in \S\ref{sec:emitters}.

\begin{figure}[ht]
\centering
\includegraphics[width=2.5in]{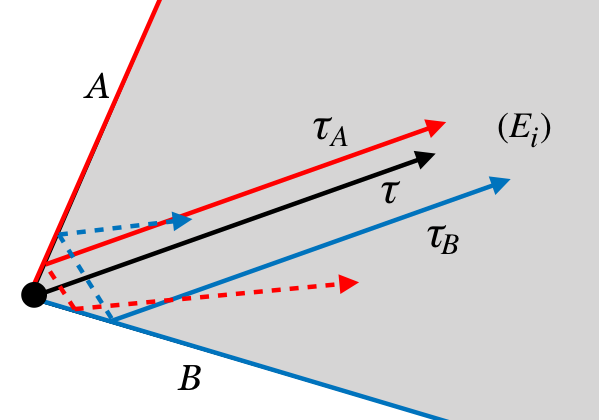}
\caption{Since $A$ and $B$ are adjacent, there are trajectories $\tau_A, \tau_B$ based at points on $A$ and $B$ (red and blue, respectively) whose bounce sequences have tails that match that of a singular trajectory $\tau$ (black) based at the vertex between them.}
\label{fig:adjacency-idea}
\end{figure}

\subsection{Adjacency in convex polygons}

We begin with the simpler case of {\em strictly convex} polygons, i.e. those where all interior angles are in $(0,\pi)$.  For the remainder of the section, the modifier ``strictly'' will be assumed. First, we show that we are actually able to detect convexity.

\begin{proposition}
A polygonal billiard table $P$ is convex if and only if $E_i E_j \in \L$ for each $E_i \neq E_j \in \A$.
\end{proposition}

\begin{proof}
If $P$ is (strictly) convex, then it contains all line segments between points in $P$. In particular, given any two edges $E_i$ and $E_j$, $P$ contains a nonsingular line segment in the interior of $P$ connecting interior points of the two, hence $E_i E_j \in \L$. (Observe that this is where we need strict, rather than regular, convexity).

Conversely, if $P$ is not strictly convex then it contains some reflex or straight angle. The edges forming this angle have no straight line between them lying in the interior of $P$, so there is no trajectory going from one to the other, and so $E_i E_j \notin \L$.\end{proof}

The following proof uses common prefixes to reconstruct adjacencyof edges from $\B(P)$ for convex polygons. This proof contains the main ideas of the proof of the general case, which we cover in the next subsection.

\begin{theorem}\label{thm:convex_adjacent}
Let $P$ be a convex polygon. Then
edges $A$ and $B$ are adjacent
if and only if
$\{A,B\}$ is a realizable set of common prefixes.
\end{theorem}

\begin{proof}
Suppose first that $A$ and $B$ are adjacent at vertex $p$. Choose a direction $\theta$ such that the trajectory $\tau = \tau_+(p,\theta)$ is nonsingular. Let $(E_i) = \B_+(p,\theta)$.

For each $N$, consider the corridor associated to $E_1 \ldots E_N$. By Lemma \ref{lem:neighborhood_in_corr}, there exists an $\varepsilon$-tubular neighborhood about $\tau$ lying in this corridor.
As $A$ and $B$ meet at $p$, we can find $a \in A$ and $b \in B$ that are within $\varepsilon$ of $p$.
Thus the trajectories $\tau_+(a,\theta)$ and $\tau_+(b,\theta)$ remain $\varepsilon$-close to $\tau$ in $\D_P((E_i)_{i=1}^N)$, so they are nonsingular and realize the bounce words $A E_1 \ldots E_N$ and $ B E_1 \ldots E_N$, respectively.
Hence as this holds for arbitrarily high $N$, we see that
\[(A, E_1, E_1, \ldots), (B, E_1, E_2, \ldots) \in \overline{\B_+(P)},\]
i.e., $\{A,B\}$ is a realizable set of common prefixes for the aperiodic bounce sequence $(E_i)$.

Conversely, suppose that $\{A,B\}$ is a set of realizable common prefixes for an aperiodic bounce sequence $(E_i) \in \B_+(P)$.
{Then there are sequences of trajectories with basepoints $a_k$ and $b_k$ on $A$ and $B$, respectively, that approximate $(E_i)$ (Figure \ref{fig:adjacency_corridor}). Theorem \ref{thm:GKT} tells us that there is exactly one trajectory realizing any aperiodic bounce sequence in $\B_+(P)$.
Thus, these sequences of trajectories converge to the same (unique) trajectory $\tau$ that realizes $(E_i)$. Hence the limit points $a = \lim a_k$ and $b = \lim b_k$ of the basepoints for the trajectories are collinear along $\tau$. Since $P$ is convex, $\tau$ intersects $\partial P$ at most twice. Since $\tau$ is nonsingular in the forward direction, it hits $\partial P$ in the interior of an edge that is neither $A$ nor $B$. Therefore at the other intersection of $\tau$ with $\partial P$, $\tau$ must intersect both $A$ and $B$. Thus $a$ and $b$ must coincide, which means that $A$ and $B$ share a point. Therefore $A$ and $B$ are adjacent.}
\end{proof}

\begin{figure}[ht]
\centering
\includegraphics[width=4in]{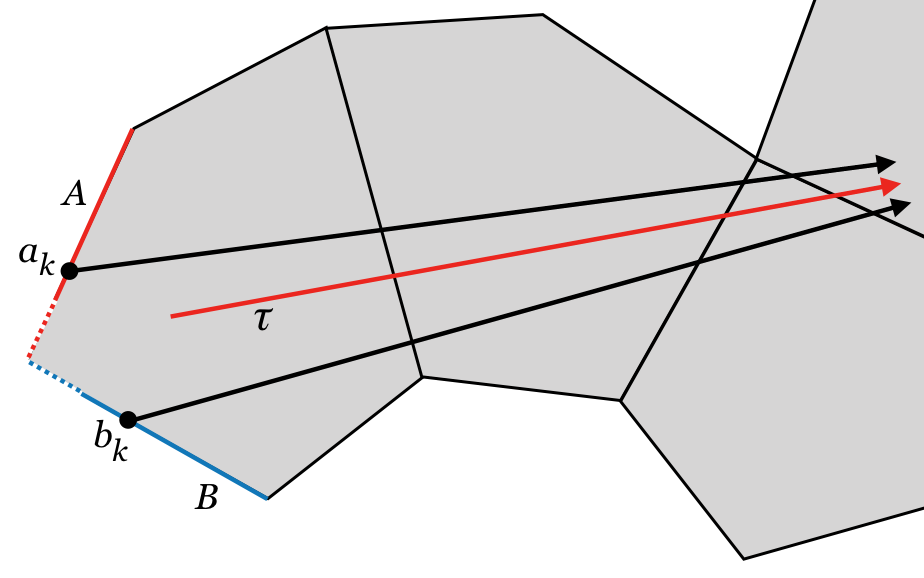}
\caption{Approximating an infinite trajectory $\tau$ by rays based on edges $A$ and $B$, as in the proof of Theorem \ref{thm:convex_adjacent}.}
\label{fig:adjacency_corridor}
\end{figure}

\subsection{Adjacency in non-convex polygons}

The majority of the proof of Theorem \ref{thm:convex_adjacent} involves showing that if two edges are common prefixes, then there exist points on those edges that are collinear. This conclusion about common prefixes containing collinear points is proven without using the hypothesis of convexity, and so the statement also applies to the general non-convex setting. In fact, this collinearity condition also holds even when the set of common prefixes is not realizable.

We record this fact for use in the general case:

\begin{lemma}\label{lem:collinear_limiting_points}
Suppose that $\{A_1, \ldots, A_n\}$ is a set of common prefixes for an aperiodic
bounce sequence
$(E_i)\in \overline{\B_+(P)}$. For each $j=1, \ldots, n$ and each $N>0$, let
\[(a_N^j, \theta_N^j) \in A_j \times S^1\]
denote a pair that realizes $A_j E_1 \ldots E_N$. Then we have the following.
\begin{itemize}
\item For each $j$ the sequence $\theta_{(N)}^j$ converges to some $\theta \in [0, 2\pi)$.
  \item For each $j$, either $a_{(N)}^j$ converges to a point $a_j \in A_j$ or $A_j$ is at angle $\theta$ with the horizontal.
  \item Take these limit points $a_j$, setting $a_j$ to be an arbitrary point of $A_j$ in the case where $A_j$ is at angle $\theta$. Then there exists a point $q$ on $E_1$ such that the points $\{q, a_1, \ldots, a_n\}$ all lie on a line contained in $P$.
\end{itemize}
\end{lemma}

\begin{proof}
All of the sequences of trajectories that approximate $(E_i)$ converge to the same (unique) ideal trajectory $\tau$ guaranteed by Lemma \ref{lem:ideal_uniqueness}. Hence their angles all converge to $\theta$ and their basepoints converge to a set of collinear points $\{a_1, \ldots, a_n\}$, with the case of $A_j$ having angle $\theta$ handled as in the statement. Finally, $q$ may be taken as any point in the intersection of $\tau$ and $E_1$, which is clearly collinear (along $\tau$) with $\{a_1, \ldots, a_n\}$. \end{proof}

In the proof of Theorem \ref{thm:convex_adjacent}, after the collinearity of $a$ and $b$ is established, the hypothesis of convexity is applied to deduce the adjacency of $A$ and $B$. In the absence of a convexity hypothesis, collinearity is insufficient to detect adjacency. This means we will need to develop more sophisticated tools for decoding adjacency fin non-convex polygons.

To illustrate why collinearity alone cannot detect adjacency in the non-convex setting, consider the following example. See Figure \ref{fig:reflex_coemitters}. Suppose that edges $B$ and $C$ meet a point $p$ in a reflex angle. Orient the edges of $P$ clockwise, and let the angles of $C$ (negative) and $B$ (positive) with the horizontal be denoted by $\gamma$ and $\beta$.
Choose some nonsingular $\theta \in (\gamma, \beta)$ and let $(E_i) = \B_+(p,\theta)$. Then if $A$ is the first edge hit by $\tau_+ (p, \pi + \theta)$, we see that $A$ and $B$ are both prefixes for a trajectory with associated bounce sequence $(E_i)_{i=1}^\infty$.
However, $A$ and $B$ are not adjacent.

Moreover, in the non-convex setting
we can no longer use common prefixes to determine the adjacency of edges even when they meet in a non-reflex angle.
Suppose that $B$ and $C$ are as above, and $A$ and $D$ now meet in a non-reflex angle at point $q$.
Let $\theta$ be such that the straight line path from $q$ in the direction of $\theta$ hits $p$.
See Figure \ref{fig:reflex_abutter}.
Set $(E_i) = B_+(p, \theta)$.
Then $\{A,B,D\}$ is a set of common prefixes for $(E_i)$:
by taking points $b \in B$ and $d \in D$ close to $p$ and $q$, respectively, we can approximate $(E_i)$ by $\B_+(b,\theta)$ and $\B_+(d,\theta)$. By taking a point $a$ on $A$ very close to $q$ and $\theta'$ close to $\theta$, we can approximate $(E_i)$ by $\B_+(a,\theta')$.

\begin{figure}[ht]
  \centering
  \begin{subfigure}{.49 \textwidth}
\centering
\includegraphics[width=2.75in]{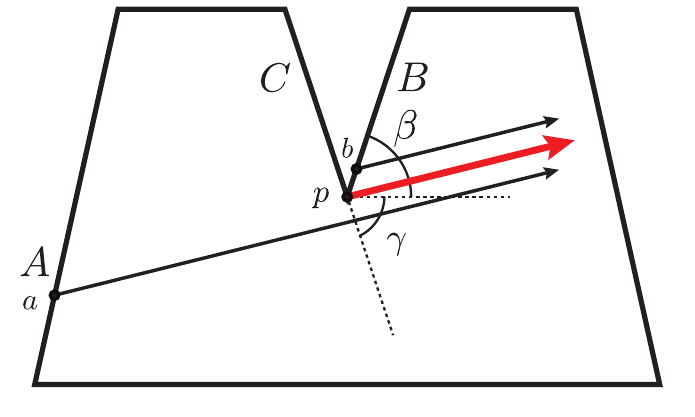}
\caption{Common prefixes $A,B$ that are not adjacent.}
\label{fig:reflex_coemitters}
  \end{subfigure}
  \begin{subfigure}{.5\textwidth}
\centering
\includegraphics[width=2.75in]{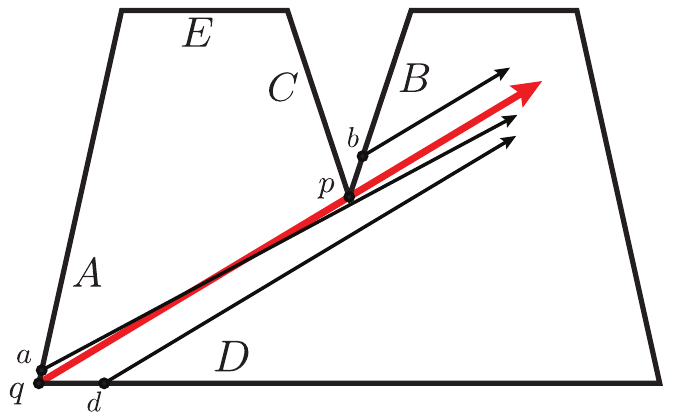}
\caption{A triple $A, B, D$ of common prefixes.}
\label{fig:reflex_abutter}
  \end{subfigure}
  \caption{The failure of collinearity to detect adjacency in non-convex polygons.}
  \label{fig:reflex_problems}
\end{figure}

Of course, even more complicated combinatorial arrangements can be imagined.\\

{\noindent \bf Grazing.} The key observation that allows us to deal with these difficulties is that in both scenarios, the ideal trajectory
\[\tau :=\tau_+(p,\pi + \theta) \cup \{p\} \cup \tau_+(p,\theta)\]
``grazes'' the vertex $p$. Approximating $\tau$ by nonsingular trajectories will then yield the bounce sequence $(E_i) \in \overline{\B(P)}$ associated to $\tau$.

Since $\tau$ passes through $p$, there are points in the interiors of $B$ and $C$ that lie arbitrarily close to $\tau$. Thus by taking points $b$ on $B$ close to $p$ we may approximate $(B, E_1, E_{2}, \ldots)$ by $\B_+(b, \theta)$, and likewise for $C$, $(C, E_0, E_{-1}, \ldots)$, and $\B_+(c,\pi + \theta)$.

We record this phenomenon in the following definition:

\begin{definition}
  We say that a pair $\{F,G\}$ \emph{grazes} an ideal trajectory if there exists an $(E_i)_{i \in \Z} \in \overline{\B(P)}$
  such that $\{E_0, E_1 \} \cap \{F,G \} = \emptyset$ and
  \begin{itemize}
\item $(F, E_0, E_{-1},  \ldots) \in \overline{B_+(P)}$ and
\item $(G, E_1, E_{2},  \ldots) \in \overline{B_+(P)}$.
  \end{itemize}
\end{definition}

Observe that this definition implies that $\{F, E_1\}$ are common prefixes for a trajectory associated to $(E_{-i})_{i=0}^\infty$, and similarly $\{G, E_0\}$ are common prefixes for a trajectory associated with $(E_{i})_{i=1}^\infty$.

In each of the scenarios in Figure \ref{fig:reflex_problems}, the pair $\{B,C\}$ grazes an ideal trajectory. Moreover, $\{A,B\}$ are common prefixes for the forward bounce sequences but are not adjacent. To rule out these cases, we want to detect and remove from consideration all forward bounce sequences that come from grazed trajectories.
To that end, if $\{F,G\}$ grazes an ideal trajectory with associated bounce sequence $(E_i)_{i \in \Z}$, then we say that the sequence $(E_i)_{i=1}^\infty$ is a
{\em grazing sequence}
for the pair of common prefixes $\{G, E_0\}$.

While not every grazing sequence comes from an arrangement exactly as in Figure \ref{fig:reflex_problems}, we can still deduce
that the limiting trajectory is singular. That is, there is geometric content to our combinatorial definition of a grazing pair:

\begin{lemma}\label{lem:what_does_abut_look_like}
  If $\{F,G\}$ grazes an ideal trajectory with associated bounce sequence
  $(E_i) \in \overline{\B(P)}$ then there exist points $f \in F$, $g \in G$, $e_0 \in E_0$, $e_1 \in E_1$ such that $\{f,g, e_0, e_1\}$ all lie on a line contained in $P$.
\end{lemma}
\begin{proof}
  Apply Lemma \ref{lem:collinear_limiting_points} to the pairs of common prefixes $\{F,E_0\}$ and $\{G, E_1\}$
  so as to get points $f \in F, g \in G$ and $e_i, e_i' \in E_i$ for $i=0, 1$
  such that $\{f, e_0, e_1\}$ are collinear and $\{g, e_0', e_1'\}$ are collinear.
  Now since $(E_i)_{i \in \Z}$ is approximated by nonsingular bi-infinite trajectories, we see by Corollary \ref{cor:realizability} that the limiting directions $\theta$ and $\theta'$ for the two trajectories must be opposite, i.e. $\theta=\pi + \theta'$.

Thus $\{f,g,e_0, e_1\}$ are all collinear.
\end{proof}

In the convex setting, this collinearity condition implies that the only grazing sequences are ones having associated ideal trajectories that are doubly singular, since a line can intersect $\partial P$ at most twice when $P$ is convex. In particular, a trajectory is grazing if and only if it is doubly singular, so the sequence of edges encountered from a vertex is realizable if and only if it is not grazing. That is to say, non-grazing is already contained in the hypothesis of realizability in Lemma \ref{thm:convex_adjacent}.

Our observations about grazing trajectories and sequences give us a way to generalize our criterion for adjacency to the non-convex setting.

\begin{theorem}\label{thm:nonconvex_adjacent}
  Let $P$ be a polygon. Then
  edges $A$ and $B$ are adjacent
  if and only if
  $\{A,B\}$ is a pair of common prefixes for a non-grazing bounce sequence.
\end{theorem}
\begin{proof}
  Suppose first that $A$ and $B$ are adjacent at a vertex $p$. Choose a nonsingular direction $\theta$ such that $\pi+\theta$ points outside of $P$. Let $(E_i)=\B_+(p,\theta)$, and let $e_1$ be the intersection point of $\tau_+(p,\theta)$ and $E_1$.

  We now show that $(E_i)$ is a
  non-grazing
  sequence. If $(E_i)$ were grazing, by Lemma \ref{lem:what_does_abut_look_like}, there would be some point $e_0$ in another edge $E_0$ (different from  $E_1, A$ and $B$) that is collinear with the line through $p$ and $e_1$. But $\theta$ was chosen to be nonsingular and $(p,\pi+ \theta)$ points outside of $P$, hence $(E_i)$ cannot be grazing.

  For the other direction, suppose that $\{A, B\}$ is a pair of common prefixes for a
  non-grazing
  sequence $(E_i)$.
  By Lemma \ref{lem:collinear_limiting_points},
  there exist points $a$, $b$ and $e_1$ on $A$, $B$ and $E_1$, respectively, which all lie on a line $L$ contained entirely in $P$. Moreover, since $e_1$ is on the interior of $E_1$, at least one of $a$ and $b$ is a vertex, for edges separate the interior of $P$ from its exterior and $L$ lies entirely inside of $P$.

  Without loss of generality, suppose that $a$ is a vertex of $A$. For convenience, we suppose that $L$ is horizontal, and that $A$ lies entirely inside the closed half-space above $L$. See Figure \ref{fig:adj_cases}. Let $\alpha \in [0, \pi)$ be the counterclockwise angle that $A$ makes with $L$. The interior of $P$ lies clockwise from $A$. Let $C$ denote the edge of $P$ that is incident to $A$ at $a$. Our goal is to show that $B=C$. We will do so by considering several cases, according to the counterclockwise angle $\gamma \in [0, 2\pi)$ that the edge $C$ makes with the horizontal. The cases for the placement of $C$ are the numbered segments in Figure \ref{fig:adj_cases}.

\begin{figure}[!ht]
\centering
\includegraphics[width=5.8in]{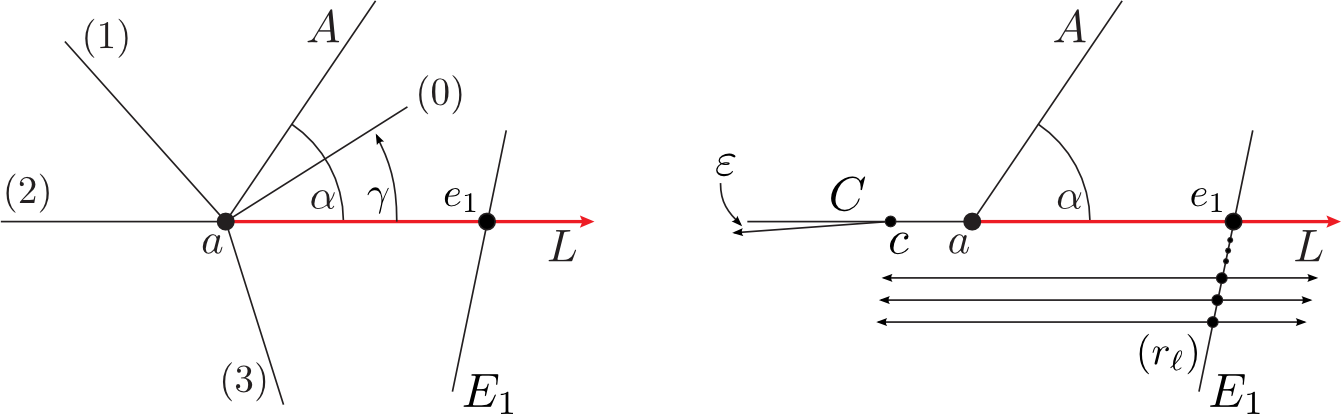}
\caption{The cases in the proof of Theorem \ref{thm:nonconvex_adjacent} (left), and the construction used in Case 2 (right).}
\label{fig:adj_cases}
\end{figure}

\begin{itemize}
\item {\em Case 0:} If $\gamma \in (0, \alpha)$ then $L$ does not lie entirely inside of $P$, a contradiction.

\item {\em Case 1:} If $\gamma \in (\alpha, \pi)$ then it is easy to see that $\{A,C\}$ grazes an ideal trajectory containing $L$, contradicting our assumption that the sequence was non-grazing.

\item {\em Case 2:} If $\gamma = \pi$ then $C$ is collinear with $L$. See the right of Figure \ref{fig:adj_cases}. Take a sequence of points $(r_\ell)$ on $E_1$ converging to $e_1$ from below such that $\tau(r_\ell, \pi)$ are all nonsingular. Then let $(F'_j)_{j\in \Z}$ be the limit of $\B(r_\ell, \pi)$, and set $(F_j)$ to be the reverse of $F'_j$, i.e. $F_j = F'_{-j}$. Observe that for all $i \ge 1$, we have $F_i = E_{i-1}$.

  Now for any $N$, there exists $c \in C$ and small enough $\varepsilon$ such that $\tau_+(c, \pi+\varepsilon)$ realizes the word $C F_{-1} \ldots F_{-N}$. In particular, this implies that
  \[(C, F_{-1}, \ldots) \in \overline{\B_+(P)}.\]
  But since $A$ is a prefix for
  $(E_i)_{i=1}^\infty$ we see that
  \[(A, E_1, E_2, \ldots) = (A, F_0, F_1, \ldots) \in \overline{\B_+(P)}\]
  hence the pair $\{A,C\}$ grazes the ideal trajectory associated to $(F_{j+1})_{j \in \Z}$.

  But we assumed that $(E_i)$ was non-grazing. Therefore $\gamma \neq \pi$.

\item {\em Case 3:} The possibility that remains is that $\gamma > \pi$. Now since $L \subset P$, this implies that no other edges of $P$ may intersect $L$. Since $b \in B \cap L$, we see that $B=C$ and conclude that $A$ and $B$ are adjacent.\end{itemize}
  This completes the proof. \end{proof}

  \subsection{Forward and full bounce spectra}\label{sec:full_determines_forward}

  We have used both one--way--infinite and bi-infinite trajectories in order to reconstruct adjacency, in the guise of common prefixes and grazing sequences, respectively. However, we originally framed our reconstruction problem only in terms of $\B(P)$. Therefore we need to show that $\B(P)$ determines $\B_+(P)$ so that we have access to information encoded one--way--infinite bounce sequences.

  It is clear that the forward tail of each sequence in $\B(P)$ lies in $\B_+(P)$. As explained in Remark \ref{rem:singular-word}, the converse is not true. For example, consider a point $(p,\theta) \in P \times [0, 2\pi)$ such that the forward trajectory is nonsingular but the backwards trajectory hits a vertex. Then $\B_+(p,\theta)$ is an element of the forward bounce spectrum but does not arise as the restriction of any element of $\B(P)$.

While this example shows that there is not a direct route to recovering $\B_+(P)$ from $\B(P)$, we may instead show that the latter determines the former by passing through their respective closures.

\begin{lemma}\label{lem:determining_forward_closure}
For any $P$, the forward bounce sequence $(E_i)_{i=1}^{\infty} \in \overline{\B_+(P)}$ if and only if there exists a full bounce sequence $(E_i)_{-\infty}^{\infty} \in\overline{\B(P)}$ whose tail is $(E_i)_{i=1}^{\infty}$.
\end{lemma}

\begin{proof}
One direction is clear: if $(E_i)_{-\infty}^{\infty} \in\overline{\B(P)}$, then for each $N$ take a nonsingular trajectory $\tau(p, \theta)$ which realizes $(E_i)_{i=-N}^{N}$. Then $\tau_+(p, \theta)$ clearly realizes $(E_i)_{i=1}^{N}$. Hence as this holds for every $N$, we have that $(E_i)_{i=1}^{\infty} \in \overline{\B_+(P)}$.

Conversely, suppose that $(E_i)_{i=1}^{\infty}  \in \overline{\B_+(P)}$. For any $N$, construct the corridor for $E_{1} \dots E_{N}$ and take some nonsingular forward trajectory $\tau_+$ lying in this corridor. By Lemma \ref{lem:neighborhood_in_corr}, this corridor contains a tubular neighborhood about $\tau_+$. In particular, it contains some nonsingular bi-infinite trajectory $\tau_N$ that lies in the neighborhood, hence the corridor, and realizes $E_{1} \dots E_{N}$. Choose such a $\tau_N$ for each $N$. Then by compactness of $\B(P)$ the sequence $\B(\tau_1), \B(\tau_2), \ldots$ converges to some $(F_i)_{-\infty}^{\infty}$, and by construction we must have that $F_i = E_i$ for every $i \ge 1$.
\end{proof}

Since $\B(P)$ clearly determines $\overline{\B(P)}$, we only have left to show that for a given $(E_i) \in \overline{\B_+(P)}$, we can distinguish when it is actually realized. We can do so by considering the maximal set of common prefixes for $(E_i)$.

\begin{theorem}
\label{thm:detecting_forward_realizability}
A sequence $(E_i)_{i=1}^\infty \in \overline{\B_+(P)}$ is realizable if and only if for every $n \ge 1$ the only prefix for $(E_i)_{i=n+1}^\infty$ is $E_{n}$.
\end{theorem}
\begin{proof}
Suppose first that $(E_i)$ is realizable, that is $(E_i) \in \B_+(P)$. Take $(p,\theta)$ such that \mbox{$\B_+(p,\theta) = (E_i)$}. Then the trajectory meets the edge $E_n$ in its interior and transversely at a point $q$. By Lemma \ref{lem:collinear_limiting_points} we see that no other edge $A$ may be a prefix for $(E_i)_{i=n+1}^\infty$. If it were, then the limit point $a\in A$ would be collinear with $p$ and $q$ along a line contained in $P$. But this would force the trajectory $\tau_+(p,\theta)$ to be singular, contrary to our assumption.

Now suppose instead that $(E_i)_{i=1}^\infty$ is not realizable. Choose some sequence $(p_n, \theta_n)$ that realizes $E_1 \ldots E_n$. By compactness of $P$ we may take some subsequence limiting to a point $(p,\theta)$, which must necessarily define a singular trajectory. There are two cases:

\begin{itemize}
\item Suppose the point $p$ is not a vertex. Then since $(E_i)$ is not realizable, $\theta$ must be a singular direction for $p$. Looking at the tail of the trajectory starting from the singularity, we can reduce to the case above where $p$ is a vertex.
\item Now suppose $p$ is a vertex. Approximate $(E_i)_{i=1}^\infty$ from each side by nonsingular bi-infinite trajectories $\tau(a, \theta)$ and $\tau(b,\theta)$ lying on each side of $\tau_+(p, \theta)$. Let $A$ denote the first edge hit by $\tau_+(a, -\theta)$ and similarly for $B$ and $\tau_+(b, -\theta)$.

Since $p$ is a vertex, $A \neq B$, and by construction, $A$ and $B$ are common prefixes for $(E_i)$.
\end{itemize}
This completes the proof.\end{proof}

We therefore obtain the following relation between the infinite and bi-infinite symbolic codings.

\begin{corollary}\label{coro:full_determines_forward}
$\B(P) = \B(P')$ if and only if $\B_+(P) = \B_+(P')$.
\end{corollary}
\begin{proof}
For the forward direction, we note that $\B(P)$ determines $\overline{\B(P)}$,
which determines $\overline{\B_+(P)}$
(Lemma \ref{lem:determining_forward_closure}),
which in turns determines $\B_+(P)$
(Theorem \ref{thm:detecting_forward_realizability}).

For the reverse direction, we simply note that for a sequence $(E_i)_{i \in \Z}$ to be in $\B(P)$ it suffices to show that both its negative and non-negative tails are in $\B_+(P)$.
\end{proof}

Now that we have established the relationship between the forward and full bounce spectra, we have the following theorem, as stated in the introduction.

\begin{theorem}\label{thm:adj} The adjacency of edges in a polygonal billiard table $P$ can be reconstructed from $\B(P)$. \end{theorem}

\section{Reconstructing angles}\label{sec:angles}

Now that we know how to determine adjacency from $\B(P)$, we turn to the question of reconstructing angles between adjacent pairs of edges. We will first show how to do this for rational angles, namely those of the form $\pi\cdot p/q$. We will then show how to reconstruct irrational angles from $\B(P)$ by using our technique for rational angles and applying a limiting argument.

The basic idea of the construction is shown in Figure \ref{fig:angles_idea}. We can get a coarse approximation to the angle between adjacent edges $A$ and $B$ simply by measuring the longest subword of alternating $A$s and $B$s in any bounce sequence. For example, the bounce sequence corresponding to the trajectory in Figure \ref{fig:angles_idea} (a) has $ABABA$ as a subword, which tells us that the angle at the vertex is at most $\pi/4$. This is because there are at least four copies of the angle sitting within the straight angle of $\pi$ carved out by the trajectory. By unfolding the table around and around the vertex, and using parallel trajectories that pass very close to the vertex, we can measure the angle as accurately as we like. For example, in Figure \ref{fig:angles_idea} (b) we use a total of four trajectories, and wrap the table around the vertex twice, cutting through a total of $18$ edges, to measure the vertex angle as $2\pi/9$.

\begin{figure}[ht]
  \centering
  \includegraphics[width=0.45\textwidth]{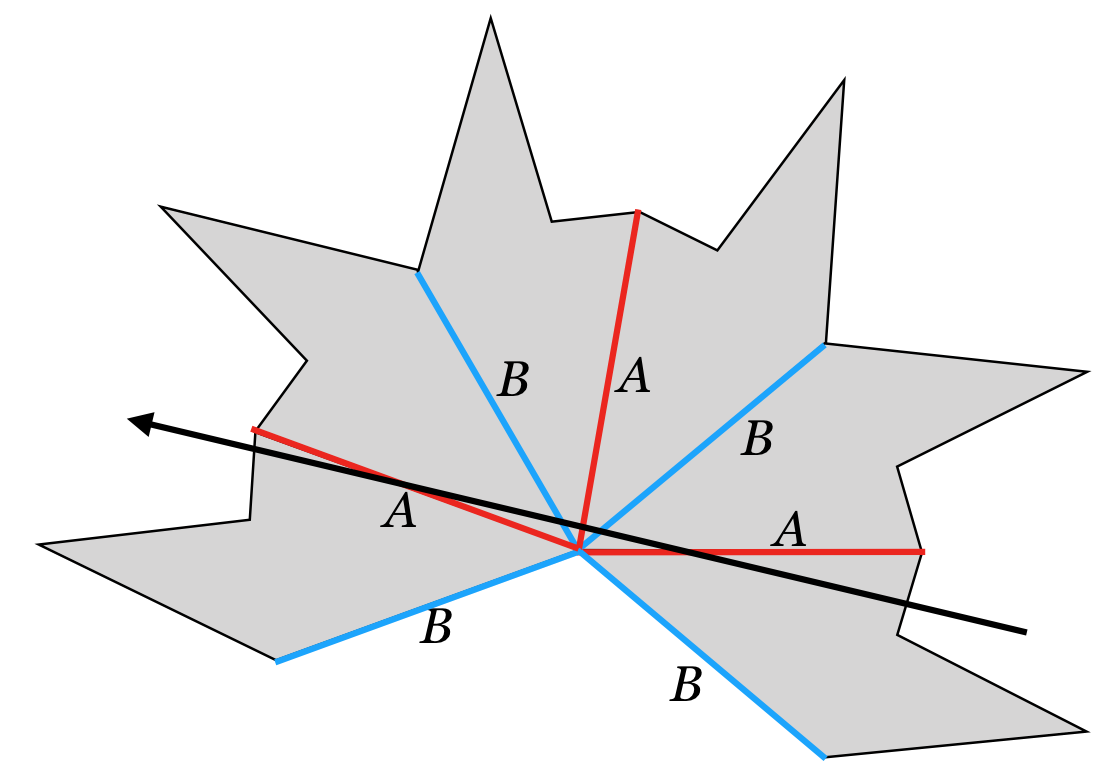}
  \includegraphics[width=0.53\textwidth]{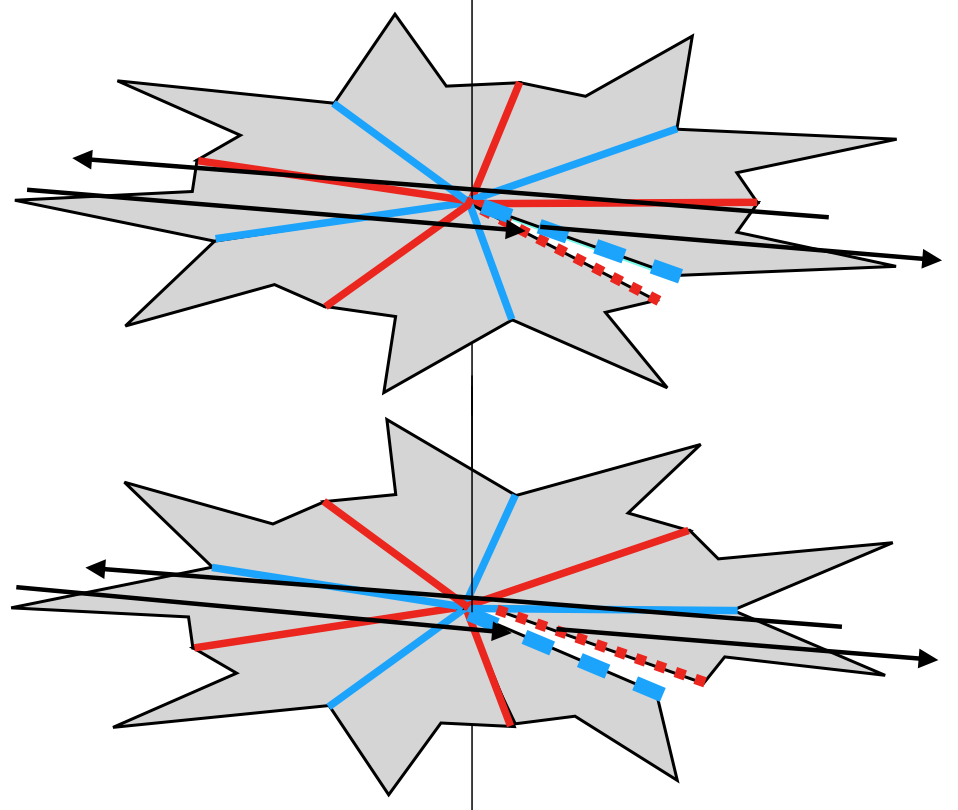}
  \caption{The basic idea of measuring angles: (a) The more edge copies a trajectory can cut through, the smaller the angle. (b) This is a side view. The blue dashed edges are identified, as are the red dotted edges. To get the most precise possible measurement of the angle, we unfold the table around the vertex until it matches up with itself again. We detect the angle using parallel trajectories that, together, ``wrap around'' the vertex.}
  \label{fig:angles_idea}
\end{figure}

In \S\ref{subsec:right}, we show how to detect right-angled tables, that is, tables where each angle is either $\pi/2$ or $3\pi/2$. In Proposition \ref{prop:hear_pi/2}, we use the``retro-reflecting'' property of right angles to detect right angles from information in $\B(P)$. In Proposition \ref{prop:hear_3pi/2_gen}, we give a characterization of angles of $3\pi/2$ using an ``unfolding'' method that generalizes to all angles of the form $\pi\cdot p/q$ for $p, q\in\mathbb{N}$.

In \S\ref{subsec:general_angles}, we give a general ``unfolding" construction to detect rational angles from information in $\B(P)$. In this construction, we develop the polygon around and around the specified vertex, until an unfolded edge matches up with an edge of the original polygon. This requires $p$ circuits of the vertex, for a total cone angle of $2p\pi$. We then construct trajectories on this unfolding and analyze them in order to recover $p$ and $q$. Two key elements of the construction are Sturmian sequences and billiard trajectories on the square table.

In \S\ref{subsec:oracle} we show that angles of an (arbitrary) polygon can be reconstructed from information in $\B(P)$. We prove this as Theorem \ref{thm:hearing_angles}.

In the Appendix \S \ref{app:effective_angles}, we give another explicit example of our unfolding construction, which the reader may use as a reference throughout the section in order to illuminate our discussion.\\

\noindent {\bf Conventions.} Throughout this section, we assume that we are trying to determine the angle between adjacent edges $A$ and $B$, whose adjacency has been determined through the bounce spectrum by Theorem \ref{thm:nonconvex_adjacent}. We further assume that edge $A$ is horizontal, that along edge $A$ the polygon lies above $A$, and that edge $B$ meets edge $A$ at the left endpoint of $A$. ``The vertex'' always refers to the point $p$ at which edges $A$ and $B$ meet.

\subsection{Detecting right angles}\label{subsec:right}

As shown in \S\ref{subsec:rect}, the information in the bounce spectrum cannot distinguish between two right-angled billiard tables that are images of each other under an edge-parallel stretching. Thus, it is important to know when we have a right-angled table. In this section, we will show how to detect a right-angled table from information in $\B(P)$, first by detecting angles of $\pi/2$ (Proposition \ref{prop:hear_pi/2}) and then by detecting angles of $3\pi/2$ (Proposition \ref{prop:hear_3pi/2_gen}).

Note that in the rest of the paper we give trajectories in the form $\tau(p,\theta)$, where $\theta$ is an angle in $[0,2\pi)$. In this section, it is sometimes convenient for us to specify trajectories using the form $\tau(p,\mathbf{v})$, where $\mathbf{v}=[x,y]$ is a vector that specifies the direction of the trajectory.

  \begin{lemma}[Retro-reflecting and splitting properties]\label{lem:retro} For perpendicular edges, we have the following simple relationships between incoming and outgoing directions (Figure \ref{fig:retro}):
\begin{enumerate}
\item (Retro-reflecting) If the angle between $A$ and $B$ is $\pi/2$, a trajectory $\ldots AB \ldots$ entering edge $A$ with direction  $\mathbf{v}$ leaves edge $B$ with direction  $-\mathbf{v}$.
\item (Splitting) Let $x,y>0$. If the angle between $A$ and $B$ is $3\pi/2$,
  a trajectory entering $A$ with direction $[-x, -y]$, and a trajectory entering $B$ with direction $[x,y]$, are parallel after the bounce with outgoing direction $[-x,y]$. Going the other way, parallel outgoing trajectories with direction $[-x,y]$ coming from edges $A$ and $B$ that meet at an angle of $3\pi/2$ must have met $A$ and $B$ with directions  $[-x, -y]$ and $[x,y]$, respectively.
\end{enumerate}
  \end{lemma}

  \begin{proof}
See Figure \ref{fig:retro}. For part (1), note that the same construction gives a period-6 billiard trajectory in a right triangle whose hypotenuse is perpendicular to the trajectories; see \cite{100deg}, \S 1.1. For part (2), note that the restriction that the outgoing vector is of the form $[x,-y]$, where $x,y>0$, is required so that one trajectory with that direction can come from edge $A$, and a parallel trajectory can come from edge $B$.
  \end{proof}

  \begin{figure}[ht]
\centering
\includegraphics[width=0.8\textwidth]{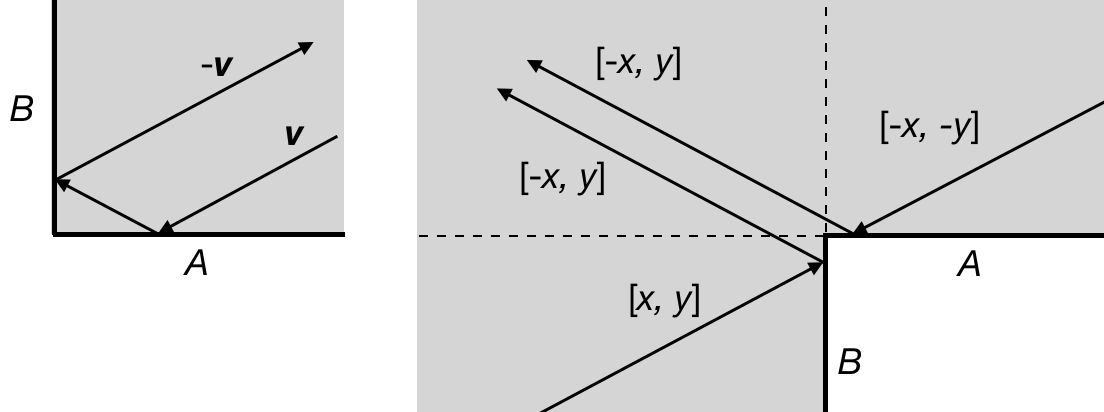}
\caption{The ``retro-reflecting''  (left, angle of $\pi/2$) and ``splitting'' (right, angle of $3\pi/2$) properties of perpendicular edges allow us to detect them in the bounce spectrum.}
\label{fig:retro}
  \end{figure}

  \begin{proposition}\label{prop:hear_pi/2}
$A$ and $B$ meet at an angle of $\pi/2$ if and only if there is an aperiodic non-grazing sequence $(E_i)_{i\in\N}$ such that
\[ E_n \cdots E_1\   A B  \ E_1 \cdots E_n\  \in\  \L\  \text{ for all }n\in\mathbb{N}.\]
  \end{proposition}

  \begin{proof}
Suppose that $A$ and $B$ meet at an angle of $\pi/2$. Consider a nonsingular one--sided trajectory $\tau$ from the vertex. An incoming trajectory parallel to $\tau$, meeting edge $A$ at a point $a$ close to the vertex, has an incoming bounce sequence of the form $(\ldots E_{-2}, E_{-1}, A)$ for some edges $E_i$. The trajectory bounces off of side $B$ and then, by Lemma \ref{lem:retro} (1), leaves parallel to $\tau$, with some outgoing bounce sequence $(B, E_1, E_2, \ldots)$. Since the incoming and outgoing trajectories are parallel and close together, and since $\tau$ is nonsingular, they will agree for some number $n$ of bounces, so $E_i = E_{-i}$ for $i=1, \ldots, n$. The closer to the vertex the trajectory meets $A$, the closer together the incoming and outgoing trajectories are, so the larger the $n$ is for which the incoming and outgoing bounce sequences agree (in reverse order). Thus, arbitrarily long such words are valid bounce sequences.

To show the other direction, suppose that $A$ and $B$ meet at an angle of $\theta$ and that there is a non-grazing sequence $(E_i)$ as above. We will show that $\theta=\pi/2$. In the development $\D_P(B, A)$, there is the original polygon $P$, and a copy $P'$ of $P$ resulting from reflection over the edge $B$ then $A$. Observe that $P'$ is equal to the rotation of $P$ about the vertex by an angle of $2\theta$ (see Figure \ref{fig:broken_line}). Consider the maximal nonsingular ray $\tau_+$ that realizes $(E_i)_{i\in\mathbb{N}}$. Since $\{A,B\}$ is a set of common prefixes for $(E_i)$, the methods of \S\ref{sec:adjacency} imply that $\tau_+$ emanates from the vertex.

Let $\tau_+'$ denote the image of $\tau_+$ under rotation by $2\theta$ (solid black in the figure). By construction, $\tau'_+$ also realizes $(E_i)_{i\in\mathbb{N}}$.

For each $n >0$, consider the corridors $\text{Corr}_1(n)$ and $\text{Corr}_2(n)$ associated with the words $E_1 \ldots E_n$ and $BA E_1 \ldots E_n$, respectively. Since $\tau_+$ realizes $E_1 \ldots E_n$, it lies in $\text{Corr}_1(n)$ for all $n$. Similarly, and since $\tau_+'$ realizes the same word but starts in $P' = r_A r_B P$, we have that $\tau_+'$ lies in $\text{Corr}_2(n)$.

Now by assumption $(\ldots , E_2, E_1, A, B, E_1, E_2, \ldots ) \in \overline{\B(P)}$ and hence by Lemma \ref{lem:ideal_uniqueness} there is a unique ideal trajectory $\tau$ associated to this bounce sequence. In particular, $\tau$ must lie in both $\text{Corr}_1(n)$ and $\text{Corr}_2(n)$ for all $n$, and hence $\tau_+$ and $\tau_+'$ must be collinear with $\tau.$

Thus since an ideal trajectory is a straight line, and $\tau_+$ and $\tau_+'$ meet at angle $2 \theta$, we see that $2\theta = \pi$. As $AB \in \L$, we see that $A$ and $B$ cannot meet at a reflex angle, and so we may conclude that the angle between $A$ and $B$ must be $\pi/2$.
  \end{proof}

  \begin{figure}[ht]
\centering
\includegraphics[width=0.6\textwidth]{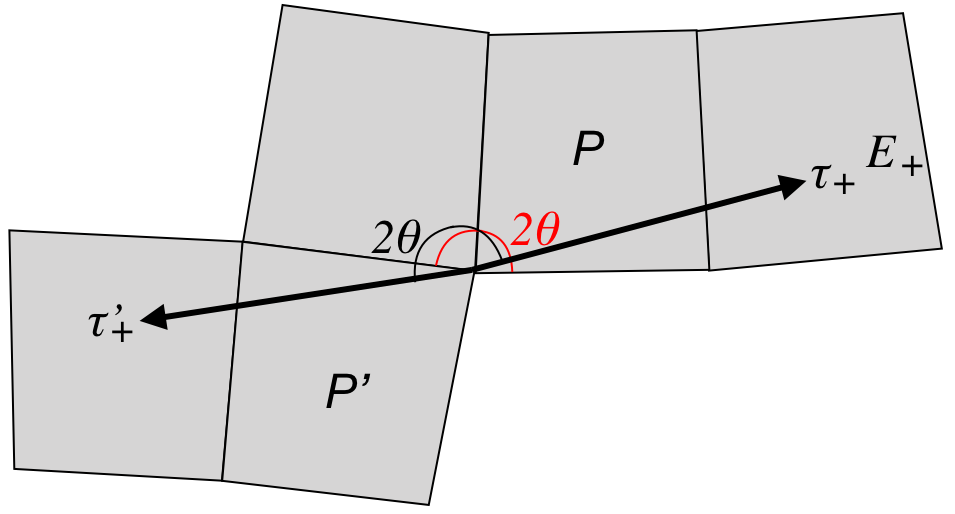}
\caption{The rays used in the proof of Proposition \ref{prop:hear_pi/2}. If a trajectory corresponding to a valid bounce sequence tracks both rays, then it must be straight, so $2\theta = \pi$ and $\theta=\pi/2$.}
\label{fig:broken_line}
  \end{figure}

  It is possible to detect an angle of $3\pi/2$ using the ``splitting'' property, similar to the method of Proposition \ref{prop:hear_pi/2}. Instead, we will use an ``unfolding'' method, in order to introduce the general theory that follows.

\begin{proposition}
\label{prop:hear_3pi/2_gen}\label{prop:3pi/2gen}
For adjacent edges $A$, $B$, the interior angle between $A$ and $B$ is $3\pi/2$ if and only if there exist grazing sequences $(E^1_i)_{i\in\Z}, (E^2_i)_{i\in\Z}, (E^3_i)_{i\in\Z},  \in \overline{\B(P)}$ such that the six words
\begin{align}
E^1_{-n}\  \ldots E^1_{-2}\  E^1_{-1}\  &\phantom{A\ } E^1_1\  E^1_2\  \ldots,\  E^1_n \label{tau1} \\
E^1_n\  \ldots E^1_{2}\  E^1_{1}\  &B\  E^2_{-1}\  E^2_{-2}\  \ldots\  E^2_{-n} \label{tau2}\\
E^2_{-n}\  \ldots E^2_{-2}\  E^2_{-1}\  &A\  E^2_1\  E^2_2\  \ldots\  E^2_n \label{tau3}\\
E^2_n\  \ldots E^2_{2}\  E^2_{1}\  &\phantom{A\ } E^3_{-1}\  E^3_{-2}\  \ldots\  E^3_{-n} \label{tau4}\\
E^3_{-n}\  \ldots E^3_{-2}\  E^3_{-1}\  &B\  E^3_1\  E^3_2\  \ldots\  E^3_n\label{tau5} \\
E^3_n\  \ldots E^3_{2}\  E^3_{1}\  &A\  E^1_{-1}\  E^1_{-2}\  \ldots\  E^1_{-n} \label{tau6}
\end{align}
are in $\L$ for all $n\in\N$.
\end{proposition}

Here we label each of the six sequences so that we may refer to them in the proof below.

  \begin{proof}
First, suppose that the angle between $A$ and $B$ is $3\pi/2$. We will construct the required sequences. As shown in Figure \ref{fig:3pi2-unfold}, we can unfold the angle of $3\pi/2$ around the vertex. We start with the picture in the upper left, with edge $A$ horizontal and the polygon above it, meeting $B$ at its left endpoint as usual. In the Figure, we label the edges $A_1$ and $B_1$ to emphasize that it is the starting picture.

We unfold across edge $B_1$, creating a second copy of the table (upper right). We then unfold across edge $A_2$, creating a third copy of the table (lower right). We unfold across edge $B_3$, creating a fourth copy of the table (lower left). Now if we unfold across the edge labeled $A$, we obtain a copy of the table that is in the same orientation as the starting picture, so this edge is again $A_1$, and four copies is all we need. This unfolding yields a total cone angle of $3\cdot 2\pi$ around the vertex.

Consider a trajectory $\tau^1$ (solid black in Figure \ref{fig:3pi2-unfold})  that grazes $p$ and is nonsingular outside of $p$, and let $\B(\tau^1) = (E^1_i)_{i\in\Z}$ with $E^1_{-1}$ and $E^1_{1}$ being the edges hit before and after grazing $p$ (the upper-left corner of Figure \ref{fig:3pi2-unfold}).

Parallel to $\tau^1$ and above it is a trajectory $\tau_1$ whose direction is $[-x,-y]$ for some $x,y>0$ (the dashed black trajectory in Figure \ref{fig:3pi2-unfold}). As $\tau_1\to\tau^1$, $\B(\tau_1)$ agrees with $\B(\tau^1)$ for arbitrarily many bounces, giving us the bounce sequences \eqref{tau1}.

\begin{figure}[!h]
\centering
\includegraphics[width=0.7\textwidth]{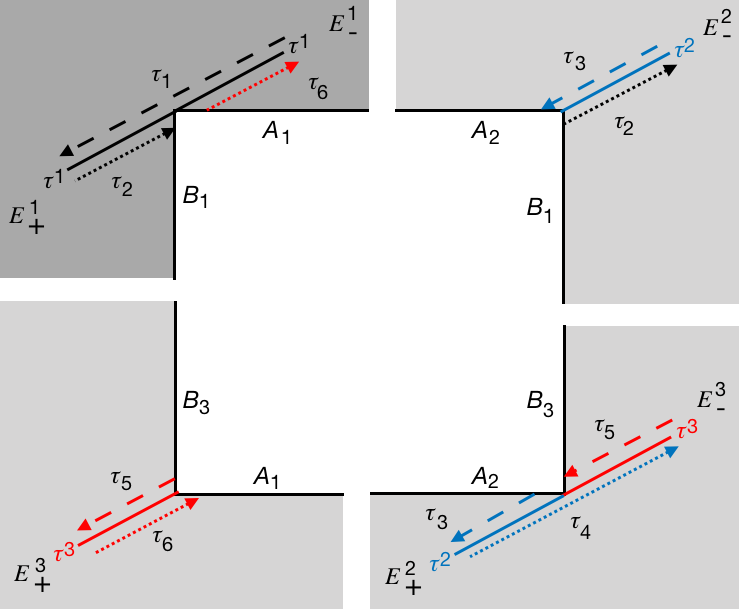}
\caption{Starting with the upper-left dark polygon, we unfold the angle of $3\pi/2$ around the vertex, yielding 4 copies for a total angle of $3\cdot 2\pi$. We use the six nearby parallel trajectories $\tau_1, \ldots, \tau_6$ and their associated bounce sequences to uniquely characterize the vertex angle as $3\pi/2$.}
\label{fig:3pi2-unfold}
\end{figure}

Parallel to $\tau^1$ and below it is a trajectory $\tau_2$ (dotted black in Figure \ref{fig:3pi2-unfold}) whose direction is $[x,y]$ for some $x,y>0$. As $\tau_2\to\tau^1$, $\B_-(\tau_2)$ and $\B_+(\tau^1)$ agree for arbitrarily many bounces. Since $\tau_2$ is below $\tau^1$, which passes through $p$, $\tau_2$ passes through edge $B$. Thus we get a $B$ in $\B(\tau_2)$, and we must unfold the table across edge $B$, yielding the picture in the upper-right corner of Figure \ref{fig:3pi2-unfold}.

Now, above $\tau_2$ and parallel to it, passing through $p$ in the unfolded copy of the table, is a singular trajectory $\tau^2$ (solid blue in Figure \ref{fig:3pi2-unfold}), with $\B(\tau^2) =  (E^2_i)_{i\in\Z}$, where $E^2_{-1}$ and $E^2_{1}$ are the edges hit before and after grazing $p$. As $\tau_2\to\tau^2$, $\B_+(\tau_2)$ agrees with $\B_-(\tau^2)$ for arbitrarily many bounces, so $\B_+(\tau_2) = \B_-(\tau^2)$. This construction yields the bounce sequences \eqref{tau2}.

Parallel to $\tau^2$ and above it is a trajectory $\tau_3$ (the dashed blue trajectory in Figure \ref{fig:3pi2-unfold}) whose direction is $[-x,-y]$ for some $x,y>0$. As above, as $\tau_3\to\tau^2$, $\B_-(\tau_3)$ and $\B_-(\tau^2)$ agree for arbitrarily many bounces. Since $\tau_3$ is above $\tau^2$, it passes through edge $A$, so the next letter in $\B(\tau_3)$ is $A$, and we must unfold the table across edge $A$, yielding the lower-right picture in Figure \ref{fig:3pi2-unfold}.

Now, below $\tau_3$ and parallel to it, passing through $p$ in the unfolded copy of the table, is another singular trajectory, which we may consider to be the positive half of $\tau^2$, since together the two halves comprise the limiting trajectory as $\tau_3$ approaches $p$, which is the entire grazing trajectory $\tau^2$. So as above, by construction, as $\tau_3\to\tau^2$, $\B_+(\tau_3)$ and $\B_+(\tau^2)$ agree for arbitrarily many bounces, yielding the sequences \eqref{tau3}.

The sequences \eqref{tau4} and \eqref{tau5} are obtained similarly to sequences \eqref{tau2} and \eqref{tau3}, respectively.

Finally, below $\tau^3$ and parallel to it is a trajectory $\tau_6$ (the dotted red trajectory in the figure) whose direction is $[x,y]$ for some $x,y>0$. By construction, as $\tau_6\to\tau^3$, $\B_-(\tau_6)$ and $\B_+(\tau^6) = E^6_-$ agree for arbitrarily many bounces. Since $\tau_6$ is below $p$, it passes through edge $A$, so the next edge in $\B(\tau_6)$ is $A$, and we must unfold the table across edge $A$. This yields a table in exactly the same orientation that we started with (we are back in the upper left corner of the figure), so the forward half of $\tau_6$ is below $\tau^1$, so $\B_+(\tau_6)$ and $\B_-(\tau^1)$ agree for arbitrarily many bounces, yielding the sequences \eqref{tau6}.

This concludes the constructive proof of the existence of the bounce sequences \eqref{tau1}--\eqref{tau6}.\\

For the other direction, suppose that the bounce sequences \eqref{tau1}--\eqref{tau6} are in $\L$ for all $n\in\N$. Given this information from $\B(P)$, we will show that the angle between $A$ and $B$ is $3\pi/2$.

By approximating the grazing sequence \eqref{tau1} by nonsingular trajectories, we see that the limiting ideal trajectory $\tau^1$ subtends an angle of $\pi$ on the side from which it is approximated.

The grazing sequences \eqref{tau2} give us an additional $\pi$ worth of angle around $p$, and tell us that one of the edges at $p$ is $B$. The grazing sequences \eqref{tau3} give us an additional $\pi$ worth of angle around $p$, and tell us that one of the edges at $p$ is $A$. The same holds for \eqref{tau4}, \eqref{tau5} and \eqref{tau6}, so there is a total of $6\pi$ of angle around the unfolded vertex $p$, which is the vertex between edges $A$ and $B$.

Since four edges ($B$, $A$, $B$, $A$) total appear in all of the grazing sequences, we know that there are four unfolded copies of the table at the vertex $p$. So the total angle around $p$ is $6\pi/4 = 3\pi/2$.
  \end{proof}

  \subsection{Detecting rational angles}\label{subsec:general_angles}

  Now we generalize the method used in Proposition \ref{prop:hear_3pi/2_gen} for angle $3\pi/2$ to any rational angle $\pi\cdot p/q$. The idea is that, given a vertex with this angle, we can unfold the table around the vertex until we are back to where we started. Assuming that $p/q$ is in lowest terms, this requires $2q$ copies of the table, for a total cone angle of $\pi \cdot p/q \cdot 2q = 2p\pi$. We can choose a grazing trajectory through the vertex that is nonsingular in the complement of the vertex, and then choose a trajectory above it, then below it, then above the next corresponding one, then below that one, etc. as in Proposition \ref{prop:3pi/2gen}. In total, we will use $2p$ trajectories to wind around around the vertex $p$ times. The number of $A$s and $B$s in the middle of the corresponding sequences determines $q$. We will see that having $2p$ trajectories, with a total of $2q$ $A$s and $B$s, encodes an angle of $\pi\cdot p/q$.

First, we will introduce several auxiliary definitions, so that we can describe the alternating $A$s and $B$s in the middle of each sequence. We call these \emph{insertions}. In Theorem \ref{thm:hearangles}, for an angle of $\pi\cdot p/q$ we will use a list of $2p$ infinite sequences; the total number of alternating $A$s and $B$s in the ``middle'' will be $2q$. The $A$s and $B$s alternate within each sequence and the next sequence picks up where the previous one left off, so if one ends $\ldots A, B, A$, the next one will start $B, A, \ldots$.

\begin{definition}
Given a sequence $(E) = \ldots E_{-2} E_{-1} E_1 E_2\ldots$, an \emph{insertion} is a finite (possibly empty) string of alternating $A$s and $B$s that are inserted between $E_{-1}$ and $E_1$.
\end{definition}

For example, in \eqref{tau1} the insertion is the empty string, and in \eqref{tau2} the insertion is $B$. The terminology \emph{insertion} is chosen by analogy with inserting base pairs into a DNA sequence.

  Now we define a function to say how many alternating $A$s and $B$s comprise the insertion in the middle of each sequence that we use to measure angles. Given a rational angle $\pi\cdot p/q$ in lowest terms, compute the bounce sequence corresponding to one period of the trajectory with slope $p/q$ on the square billiard table, where $0$ and $1$ label the horizontal and vertical edges of the table, respectively. The method for doing this is described in \cite{flatsurfaces}, \S 7, Algorithm 7.6. (Note that this finite sequence is the same as the double of the cutting sequence for the corresponding trajectory on the square torus; see Figure \ref{fig:grid}.)

We will use the convention that such a sequence always starts with $0$. Notice that this sequence consists of $2p$ 0s and $2q$ 1s. This is essentially because $p$ corresponds to the ``rise,'' so it is the number of times each of the top and bottom edges are hit, and $q$ corresponds to the ``run,'' so it is the number of times each of the left and right edges are hit; see  \cite{flatsurfaces} for details.

We will think of the bounce sequence as consisting of (possibly empty) strings of 1s, separated by single 0s. Since the sequence has $2p$ 0s, it also has $2p$ of the (possibly empty) strings of 1s.

  \begin{definition}
Define the insertion string $\ins_i(p/q)$ of alternating $A$s and $B$s for $i=1, \ldots, 2p$ as follows:
\begin{itemize}
\item Each $\ins_i(p/q)$ consists of a string (possibly empty) of alternating $A$s and $B$s. We fix the convention that $\ins_1$ starts with $B$, unless $\ins_1$ has length $0$, in which case the first nonzero $\ins_i$ starts with $B$. For all other indices, $\ins_i$ starts with whichever letter $\ins_{i-1}$ (or the the previous nonempty string)  did \emph{not} end with, considering the indices modulo $2p$.
\item The length of the string $\ins_i(p/q)$ is the length of the $i^{\text{th}}$ string of 1s in the bounce sequence corresponding to slope $p/q$ on the square billiard table.
\end{itemize}
  \end{definition}

  \begin{figure}[!h]
\centering
\includegraphics[width=0.125 \textwidth]{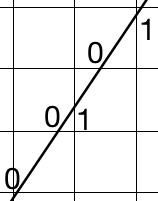} \qquad \qquad
\includegraphics[width=0.6\textwidth]{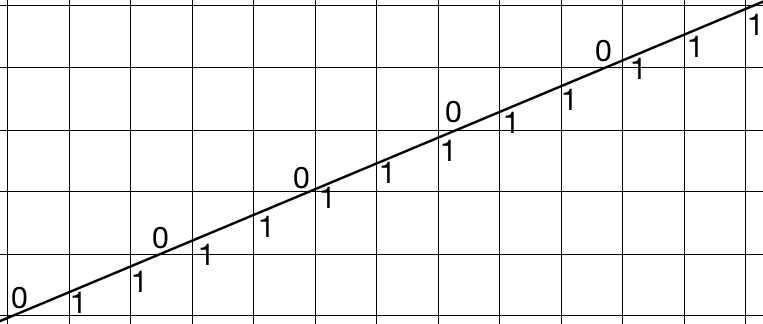}
\caption{A line of slope $3/2$ on the square grid (left), shown with one period of its associated cutting sequence, which is $00101$. The square table bounce sequence given in \eqref{bounce} is this one repeated twice. A line of slope $5/12$ on the square grid (right), shown with one period of its associated cutting sequence, which is $01101101110110111$. The square table bounce sequence given in \eqref{bouncebounce} is this one repeated twice, after a cyclic permutation.}
\label{fig:grid}
  \end{figure}

  \begin{example}
Given the angle $3\pi/2$, we compute the bounce sequence corresponding to slope $p/q = 3/2$, which is
\begin{equation}\label{bounce}
0010100101.
\end{equation}
See the left side of Figure \ref{fig:grid}. The lengths of the strings of 1s, which are also the lengths of the strings $\ins_1(p/q)$ -- $\ins_6(p/q)$, are 0, 1, 1, 0, 1, 1, respectively. So
\begin{align*}
\begin{split}
\ins_1(3/2) & = - \\
\ins_2(3/2) & = B \\
\ins_3(3/2) & = A 
\end{split}
\begin{split}
\ins_4(3/2) & = - \\
\ins_5(3/2) & = B \\
\ins_6(3/2) & =A, 
\end{split}
\end{align*}
just as in equations \eqref{tau1} - \eqref{tau6}.
  \end{example}

  \begin{example}
Given the angle $5\pi/12$  (an example that we work out in the Appendix \S \ref{app:effective_angles}), we compute the bounce sequence corresponding to slope $p/q = 5/12$, which is
\begin{equation}\label{bouncebounce}
0110111011011011101101110110110111.
\end{equation}
See the right side of Figure \ref{fig:grid}. So the strings $\ins_1(5/12)$ through $\ins_{10}(5/12)$ have lengths 2, 3, 2, 2, 3, 2, 3, 2, 2, 3, respectively. Thus $\ins_1(5/12) = B, A$; $\ins_2(5/12) = B, A, B$; \ldots; $\ins_{10}(5/12) = A, B, A$, as the reader will verify in \eqref{eq:5pi12}.
  \end{example}

  \begin{lemma}\label{lem:2q}
The total length of the insertions in $\ins_1(p/q), \ldots,  \ins_{2p}(p/q)$ is $2q$.
  \end{lemma}

  \begin{proof}
The total number of $1$s in the bounce sequence on the square table for a trajectory of slope $p/q$ is $2q$, and the lengths of the strings of 1s give the lengths of the insertions, giving the result.
  \end{proof}

  \begin{theorem}[Detecting rational angles]\label{thm:hearangles}
Suppose that edges $A$ and $B$ are adjacent in $P$ and meet in a rational angle. Then the rational angle is $p/q \cdot \pi$ in lowest terms if and only if there exist grazing sequences $(E^1_i)_{i\in\Z}, \ldots, (E^{p}_i)_{i\in\Z}  \in \overline{\B(P)}$ such that the $2p$ words
\begin{align}
E^1_{-n}\  \ldots E^1_{-2}\  E^1_{-1}\  &\overbrace{B A \ldots}^{\ins_1(p/q)}\ E^1_1\  E^1_2\  \ldots\  E^1_n \label{gtau1} \\
E^1_n\  \ldots E^1_{2}\  E^1_{1}\  &\ins_2(p/q)\  \ \ E^2_{-1}\  E^2_{-2}\  \ldots\  E^2_{-n} \label{gtau2}\\
E^2_{-n}\  \ldots E^2_{-2}\  E^2_{-1}\  \ \ \  &\ins_3(p/q)\  E^2_1\  E^2_2\  \ldots\  E^2_n \label{gtau3}\\
&\vdots \qquad\qquad\qquad \vdots \nonumber \\
E^{p-1}_n\  \ldots E^{p-1}_{2}\  E^{p-1}_{1}\  &\ins_{2p-2}(p/q)\ E^p_{-1}\  E^p_{-2}\  \ldots\  E^p_{-n} \label{gtau4}\\
E^p_{-n}\  \ldots E^p_{-2}\  E^p_{-1}\  &\ins_{2p-1}(p/q)\  E^p_1\  E^p_2\  \ldots\  E^p_n\label{gtau5} \\
E^p_n\  \ldots E^p_{2}\  E^p_{1}\ \ &\ins_{2p}(p/q)\ \ \ \ E^1_{-1}\  E^1_{-2}\  \ldots\  E^1_{-n} \label{gtau6}
\end{align} 
are in $\L_P$ for all $n\in\N$.
  \end{theorem}

The sequences $(E_i)$ are the essential tool that we use for determining sizes of angles from information in $\B(P)$. We call them \emph{matching sequences}.

\begin{definition}
The sequences $S_1, S_2, \ldots, S_{\ell} \in \overline{\B(P)}$ form a collection of \emph{matching sequences} for $\{A, B\}$ if there exists $k> 0$ such that
\begin{itemize}
\item each sequence $S_i$ has a string of alternating $A$s and $B$s of length $k-1$ or $k$, and
\item for each $n\in\N$, the finite string of $n$ symbols of $S_i$ after its string of $A$s and $B$s matches (in the reverse order) the $n$ symbols of $S_{i+1}$ before its string of $A$s and $B$s, for \mbox{$i=1, \ldots, \ell-1$}.
\end{itemize}
\end{definition}

For example, the sequences \eqref{tau1}-\eqref{tau6} are $\ell=6$ matching sequences with $k=1$.

In Theorem \ref{thm:hearangles} we use $2p$ matching sequences, with a total of $2q$ inserted $A$s and $B$s (there are $q$ of each, since they alternate), so Theorem \ref{thm:hearangles} says that the angle of the vertex is
  \[
\pi\cdot \frac{p}{q} = \pi \cdot \frac{2p}{2q} = \pi \cdot \frac{\text{$\#$ of matching sequences}}{\text{total length of inserted strings of $A$s and $B$s}}.
  \]
  For example, in the case of Proposition \ref{prop:hear_3pi/2_gen} when $p/q = 3/2$, there are a total of 6 sequences with 4 $A$s and $B$s, so the angle is $6\pi/4 = 3\pi/2$.

  Also note that it is possible for the number of inserted $A$s and $B$s to be $0$. For example, in this case where $p/q = 3/2$, we have strings of $A$, $B$, and (nothing) inserted between $E_+$ and $E_-$.

  \begin{proof}[Proof of Theorem \ref{thm:hearangles}]
Suppose first that $A$ and $B$ meet in an angle of $\pi \cdot p/q$. We will show that the $2p$ words \mbox{\eqref{gtau1}--\eqref{gtau6}} are in $\L_P$, for all $n$.

Consider a line segment $\tau_1$ making a tiny angle $\theta$ with the positive horizontal, passing just above the vertex. By construction, if $\pi \cdot p/q \leq \pi$, the segment passes through $B$. If we unfold the table around the vertex across edge $B$ and then edge $A, B,$ etc., the line segment cuts through nearly $\pi$ worth of angle. So, starting with $B$, it passes through $B, A, B, \ldots$, where the number of edges it cuts through is $\lfloor\pi/(\pi\cdot p/q)\rfloor = \lfloor q/p \rfloor$. Extending the line segment and approaching the vertex yields the bounce words \eqref{gtau1} for increasingly large values of $n$.

As in Proposition \ref{prop:hear_3pi/2_gen}, we completely unfold the table around the vertex. Since the angle at the vertex is $\pi \cdot p/q$, the complete unfolding requires $2q$ copies of the table, for a total cone angle of $2p\pi$. We construct a family of $2p$ parallel finite trajectories $\tau_1, \ldots, \tau_{2p}$, with the odd-numbered trajectories above the vertex and the even-numbered trajectories below, as in Proposition \ref{prop:hear_3pi/2_gen} and in Figures \ref{fig:3pi2-unfold} and \ref{fig:5pi12}.

The tail of the sequence corresponding to $\tau_1$ matches the head of the sequence corresponding to $\tau_2$, the tail of the sequence corresponding to $\tau_2$ matches the head of the sequence corresponding to $\tau_3$, and so on, until the tail of the sequence corresponding to $\tau_{2p}$ matches the head of the sequence corresponding to $\tau_1$. As the distance from each $\tau_i$ to the vertex goes to $0$, the tail of $\tau_i$ and the head of $\tau_{i+1}$ (considering indices modulo $2p$) agree for more and more letters. This argument yields the ``$E$'' parts (the heads and tails) of the sequences \mbox{\eqref{gtau1}--\eqref{gtau6}}.

The rest of the proof explains the number of inserted $A$s and $B$s in the ``middle'' of each sequence.

First, we show that each string starts with the letter that the previous string did \emph{not} end with. Consider the line segment $\tau_2$ that is parallel to $\tau_1$ and below the vertex. Since $\tau_1$ and $\tau_2$ are parallel and on opposite sides of the vertex, and $A$ and $B$ and all of their copies emanate from the vertex, $\tau_1$ and $\tau_2$ do not cross any of the same edges. In particular, if the last edge that $\tau_1$ crossed was (a copy of) $B$, then the first edge that $\tau_2$ crosses will be a copy of $A$. This holds for any $\tau_i$ and $\tau_{i+1}$ constructed in this manner, so in each case, the first $A$ or $B$ in the ``middle'' of the bounce sequence corresponding to such a line segment must be the letter with which the previous middle bounce sequence did \emph{not} end.

We now determine how many edges $A$ and $B$ each trajectory $\tau_i$ crosses. We are essentially seeing how many angles of $\pi\cdot p/q$ fit into an angle of $\pi$, and then seeing how many fit into the next $\pi$, and the next. This is equivalent to seeing how many copies of $p/q$ fit into 1 (and the next copy of 1, and the next). This, in turn, is equivalent to considering a line of slope $p/q$ on a square grid that starts a tiny distance $\epsilon$ to the right of the origin, and considering how many vertical line segments it intersects before intersecting the next horizontal line, and then how many vertical line segments before the next horizontal line, and the next (see Figure \ref{fig:grid}).

The number of vertical lines crossed between each horizontal line crossing corresponds exactly to the strings of 1s in the cutting sequence for a trajectory of slope $p/q$ on a square grid. We need two full periods of the cutting sequence because one period gives us an angle that is an integer multiple of $\pi$, and we need to go around a number of times that is a multiple of $2\pi$.

Conveniently, one period of the bounce sequence for the square billiard table is exactly two periods of the cutting sequence for the square grid (\cite{flatsurfaces}, Proposition 4.1). This completes the proof of the forward direction. \\

Now, suppose that the $2p$ grazing sequences \eqref{gtau1} -- \eqref{gtau6} are in $L_P$ for all $n\in\N$. We will show that:
\begin{enumerate}
\item The total angle around the unfolding of the vertex is $2p\pi$, and
\item The total number of edges that were unfolded to get the unfolding is $2q$.
\end{enumerate}
It follows from these claims that the $2p\pi$ cone angle about the vertex is tiled by $2q$ reflected copies of the table, all meeting at the same corner. Therefore the angle subtended by the corner is $2p\pi/2q = \pi \cdot p/q$, which is the angle at the vertex, as desired. We now prove each claim:

\begin{enumerate}
\item As the word \eqref{gtau1} is in $\L_P$ for each $n$, there is $\pi$ worth of angle near the vertex. Since the limit sequence of the words in \eqref{gtau1} shares its first half with the limit sequence of the words in \eqref{gtau2}, this adds an additional $\pi$ worth of angle, just as in the end of the proof of Proposition \ref{prop:3pi/2gen}.

Each additional sequence does the same, wrapping around and around the vertex. Since the last ($2p^{\text{th}}$) sequence shares its second half with the first half of the $1^{\text{st}}$ sequence, the unfolding is done after unfolding $2p$ copies of the corner, and we have a total of $2p\cdot \pi$ angle around the vertex, as desired.

\item Each time a billiard trajectory crosses an edge, we unfold across that edge so that the trajectory does not bounce, but instead is an unfolded line (segment). We call such edges {\em unfolding edges}.

  Let $\sigma_1, \ldots, \sigma_{2p}$ be the ideal trajectories realizing the bounce sequences corresponding to \eqref{gtau1}--\eqref{gtau6}. We claim that if $\tau_i$ is a trajectory approximating $\sigma_i$, then for all $i \neq j$, $\tau_i$ and $\tau_j$ will not cross any of the same unfolding edges.

  To see this, observe that $\sigma_i$ are all parallel: since the head of $\sigma_{i}$ matches the tail of $\sigma_{i+1}$, by Corollary \ref{cor:realizability} we see that they must lie at the same angle $\theta$ (up to choice of orientation for the trajectory). In particular, another application of Corollary \ref{cor:realizability} implies that for better and better approximations $\tau_i$ of $\sigma_i$ and $\tau_j$ of $\sigma_j$, all $\tau_i$ and $\tau_j$ are increasingly close to parallel. Therefore if $\tau_i$ and $\tau_j$ cross the same unfolding edge their heads/tails must match, so $j = i \pm 1$. Without loss of generality, assume $j=i+1$.

  Now by our definition of $\ins_i(p/q)$ the edges $A$ and $B$ alternate, so in particular $\tau_i$ and $\tau_{i+1}$ approximate the head of $\sigma_i$ from different sides. But then this implies that all of the unfolding edges that $\tau_i$ crosses emanate from $p$ at some angle in $(\theta, \theta + \pi)$, while all of the unfolding edges that $\tau_{i+1}$ crosses lie in the complementary sector.

  Therefore each $A$ and $B$ appearing in some $\ins_i(p/q)$ corresponds to an unfolding edge. By Lemma \ref{lem:2q}, the total number of inserted $A$s and $B$s in \eqref{gtau1} -- \eqref{gtau6}  is $2q$, so $2q$ edges were unfolded, as desired.

\end{enumerate}
This completes the proof.\end{proof}

\begin{theorem}[Detecting arbitrary angles]\label{thm:irr}
The angle between adjacent edges $A$ and $B$ is the limit, as the number of matching sequences goes to infinity, of
\begin{equation}\label{eq:irr}
\pi \cdot \frac{\text{number of matching sequences}}{\text{total length of inserted $A$s and $B$s}} \ = \ \frac{\pi}{\text{average length of insertion}}.
\end{equation}
\end{theorem}

\begin{proof}
The average length of a string of $A$s and $B$s is the average number of times a line close to the vertex cuts through unfolded copies of edges $A$ and $B$, which is the same as the number of times the vertex angle fits into $\pi$.
\end{proof}

Note that if the angle is rational, $\pi\cdot p/q$, the list of matching sequences will be periodic with period $2p$, and the calculation \eqref{eq:irr} reduces to $\frac{2p}{2q}\pi$ as in Theorem \ref{thm:hearangles}.

\begin{example}
Suppose that we have found matching bounce sequences
\begin{align*}
(\ldots E^1_{-2},\  E^1_{-1}, B, A, B,\ E^1_1,\  E^1_2,\  \ldots,) &   \\
(\ldots E^1_{2},\  E^1_{1},\ \ \ A, B,\ \  E^2_{-1},\  E^2_{-2},\  \ldots). &   
\end{align*}
in $\overline{\B(P)}$. Then the total length of inserted $A$s and $B$s is 5, in a total of 2 sequences, so the angle at the vertex is approximately $2\pi/5 = 72^\circ$.

For an idea of how much information is given by each additional sequence, consider the following two possibilities for the string length in a third sequence.

If the third sequence has $A, B, A$ as its string, our new angle approximation is $3\pi/8\approx 67.5^{\circ}.$

If the third sequence has $A, B$ as its string, our new angle approximation is $3\pi/7 \approx 77^{\circ}.$\\

For an example of the rational case, the ten sequences in \eqref{eq:5pi12} have an average string length of $2.4$, so the angle at the vertex is $\pi/2.4 = 5\pi/12 = 75^{\circ}$.
\end{example}

\subsection{Detecting arbitrary angles}\label{subsec:oracle}
Now we will describe the process of reconstructing the angles of an arbitrary polygon $P$ from its bounce spectrum $\B(P)$. We emphasize that while the information in $\B(P)$ suffices to measure the sizes of the angles of $P$, this measurement is not carried out in an effective way (e.g. in finite time using a computer with finite storage).

\begin{theorem}\label{thm:hearing_angles}
The angles of a polygonal billiard table $P$ can be reconstructed from $\B(P)$.
\end{theorem}

\begin{proof}
We will show that we can reconstruct the angles of a polygonal table $P$, given $\B(P)$. We first apply Theorem \ref{thm:adj} to determine the cyclic order of the edges of $P$.

We choose two adjacent edges $A$ and $B$ in order to determine the angle between them. We could immediately appeal to Theorem \ref{thm:hearangles} to determine the angle in the case that it is rational, but we will sketch a somewhat constructive proof to give an idea of how we would actually determine the angle.

We will attempt to find a set of matching sequences for $\{A,B\}$. We find the longest string of the form $ABAB\ldots$ or $BABA\ldots$ of alternating $A$s and $B$s that occurs in any bounce sequence, and set aside for future use all of the sequences that have such a string of this length (call it length $k$), plus the ones that have a string of alternating $A$s and $B$s of length $k-1$. Call this set of potential matching sequences $S$.

We choose one of the sequences in $S$ with a string of alternating $A$s and $B$s of length $k$, call it $\tau_1$, and a desired depth $n$ to search the head and tail of the sequences. We look at the $n$-head $E_+$ and $n$-tail $E_-$ of the sequence before and after the string of $A$s and $B$s. We wish to find a sequence ($\tau_{2p}$, the last sequence) matching its $n$-tail, as well as another sequence ($\tau_{2}$) matching its $n$-head. We search through the other sequences in $S$, and see if any of them have $n$-tail $E_+$ or $n$-head $E_-$, starting with the correct letter ($A$ for the first if the string in $\tau_1$ ended with $B$, and vice versa; $A$ for the second if the string in $\tau_1$ began with $B$, and vice versa). If so, they are (tentatively) $\tau_2$ and $\tau_{2p}$, respectively.

We then take the $n$-head of $\tau_2$ and see if we can find any sequence in $S$ with this $n$-tail, again starting its string of $A$s and $B$s with the letter that is different from what the string in $\tau_2$ ended with. If so, we call it $\tau_3$. Similarly, we take the $n$-tail of $\tau_{2p}$ and see if we can find any sequence in $S$ with this as its $n$-head, ending its string of $A$s and $B$s with the letter that is different from what the string in $\tau_{2p}$ began with. We continue in this way.

If eventually the two ends meet, with $\tau_{m} = \tau_{2p-m}$ for some $p$,
we stop. Then we count up the total number of $A$s and $B$s that arose in the middle of the sequences, and check that this is an even number, which we call $2q$. Then our best approximation of the the angle at the vertex is currently $\pi \cdot p/q$. (If not, then we must keep going, until we have an even number $2p$ of sequences and an even number $2q$ of $A$s and $B$s.)

We then repeat this process of matching $n$-heads and $n$-tails for larger and larger values of $n$, and also starting with other sequences that have strings of $k$ $A$s and $B$s. If there is some $N$ such that, for all $n>N$, the number of sequences and the total number of $A$s and $B$s are fixed at some $2p$ and $2q$, respectively, then we conclude that the angle is rational, with angle $\pi\cdot 2p/2q = \pi \cdot p/q$.

On the other hand, if increasing the value of $n$ increases the number of sequences $2p$ and the total number of $A$s and $B$s $2q$ without bound, then the angle is irrational.

By Theorem \ref{thm:irr}, we can get an approximate value for the angle $\theta$ by computing
$$\theta \approx \pi \cdot \frac{\text{number of matching sequences}}{\text{total length of insertions}}.$$
In this way we may compute the value of $\theta$ to an arbitrary degree of precision by carrying out the matching process sufficiently far. This is of course the best that we can ever do, to specify a generic irrational number. Therefore we have shown that the angles of $P$ can be recovered from $\B(P)$.
\end{proof}

\begin{corollary}\label{cor:triangle}
You can reconstruct the shape of a triangle $P$ from its bounce spectrum $\B(P)$.
\end{corollary}

\begin{proof}
A triangle is determined (up to similarity) by its angles, and by Theorem \ref{thm:hearing_angles}, you can reconstruct the angles of a billiard table from $\B(P)$.
\end{proof}

\section{The impossibility of reconstruction from finite information }\label{sec:no_finite}
\newcommand{\Pn}{\mathcal{P}_n}

In \S \ref{sec:angles}, we used an infinite amount of information from $\B(P)$ to measure the angles of a billiard table. We needed to check, for example, that for all $n\in\N$, a word of the form
\[E^1_{-n}\  \ldots E^1_{-2}\  E^1_{-1}\ B A \ E^1_1\  E^1_2\  \ldots\  E^1_n\]
existed in the bounce language.

One ambitious goal would be to try to reconstruct a table from just a finite number of finite words in $\B(P)$. This section shows that any such attempt is doomed to failure. The main idea is that in trying to reconstruct the shape of a billiard table, we start with the whole moduli space of polygons as possible candidates. We then use information from $\B(P)$ to narrow down the options of which polygon these pieces of bounce spectrum information came from, until a unique polygon (or some limited collection) remains.

One should think of each finite piece of information as restricting our set of candidate tables, and we will show that a finite total amount of information specifies only full-dimensional subset in the moduli space. Every polygon in that region has the same finite amount of information in common, and so this amount of information is insufficient to distinguish them.

In \S \ref{subsec:moduli}, we give some preliminary setup about the finite reconstruction problem and define the moduli space of polygons. In \S\ref{subsec:finitely_finite}, we prove as Theorem \ref{thm:no_finite} that it is impossible to recover the shape of a billiard table from any finite collection of finite words. The idea is that any finite bounce word persists under sufficiently small perturbations of a polygon.

\subsection{Preliminaries and the moduli space of polygons}\label{subsec:moduli}

We begin by offering two different versions of the task of determining the shape of a table from a finite number of finite words.

In the first version, we are handed some finite collection $W$ of finite words that belong to some $\L_P$ and are asked whether we are able to reconstruct the shape of $P$. We are given no further information about where $W$ came from or how it was chosen.

In this first version of the problem, we are definitely out of luck. No matter what words we receive in $W$, there are some very basic properties of $P$ that we are unable to determine. For instance, we do not even know how many sides $P$ has; the pieces of the trajectories in $P$ that contributed words to $W$ may simply have missed some number of edges. Furthermore, there will be many tables that can realize all of the words in $W$. A polygon resulting from any modifications of $P$ that do not interact with any of the pieces of trajectories that contributed to $W$ --- like adding tiny indentations or outcroppings --- will also realize all of the words in $W$. So we see right away that this version of the problem is intractable.

In the second version of the problem, we do something of a role reversal. We have all of $\B(P)$ in front of us and we wish to carefully choose a rich finite collection of bounce words for $P$, so that we could hand it over to someone else, who could then reconstruct $P$. We will show that this task, too, is hopeless, even though it might seem more promising. For instance, we would definitely be sure to include every edge label of $P$ in some word in $W$, so our partner could be confident about how many edges $P$ has. We will show that, no matter which words we choose, no finite collection $W$ is sufficiently rich to precisely specify $P$.

In order to prove this stronger impossibility result, it is helpful to think about the moduli space $\mathcal{P}_n$ of polygons with $n$ edges. We may introduce convenient local coordinates on $\mathcal{P}_n$ as follows. Let $P(z_1, \ldots, z_n)$ denote the polygon with vertices $z_i$ in $\C$ and ordered cyclically, so that edges of the polygon connect $z_i$ and $z_{i+1}$ with indices taken mod $n$. Since we are considering polygons only up to scaling and rotation, we may specify that all of our polygons have $z_1=0$ and $z_2=1$. Moreover, by reflecting over the real axis if necessary, we may also specify that $\text{Im}(z_3) \ge 0$.

For a pair of points $a$ and $b$ in $\C$, the Euclidean distance between them is $|a-b|$. More generally, let $|A-B|$ equal the infimum of $|a-b|$, where $a\in A$ and $b\in B$. For any fixed polygon $P$, we may consider the family of polygons that are ``near" to $P$.

\begin{definition}
\label{def:local_chart}
For a polygon $P=P(0,1, \ldots, z_n)$ and a fixed $\varepsilon>0$, the set of polygons
\[P_\varepsilon=\{P(0,1, \ldots, t_n): |t_i-z_i|<\varepsilon, \forall i\}\]
is an $\varepsilon$--neighborhood of $P$ in $\mathcal{P}_n$.
\end{definition}

Since edges do not come into the definition of an $\varepsilon$--neighborhood, we can extend this definition to a notion of a $\varepsilon$--neighborhood for any point set in $\C$, regardless of whether these points are the vertices of a polygon.

Observe that the $\varepsilon$--neighborhoods defined above are full-dimensional subsets of $\mathcal{P}_n$.

\subsection{Finitely many finite words cannot uniquely determine a polygon}\label{subsec:finitely_finite}

We are now ready to show that no finite collection of finite bounce words determines the shape of a polygon. We first prove a lemma showing that any single finite word in insufficient. The extension of this to finite collections is proven in Theorem \ref{thm:no_finite}.

\begin{lemma}
\label{lem:words_persist}
Let $P$ be a polygon and $w$ a finite word in its bounce language $\L_P$. Then there exists $\varepsilon>0$ such that every polygon in $P_\varepsilon$ also has $w$ in its bounce language.
\end{lemma}
\begin{proof}
Let the given word $w$ have length $n$ and let it equal $(E_i)_{i=1}^n$. Put $P$ in the coordinate plane so that $P=P(z_1, \ldots, z_m)$. Take the development $\D_P(w)$ and denote the reflected copies of $P$ as $P_0$=$P, P_1, \ldots, P_n$.
Let $\tau=\tau(p,\theta)$ be a trajectory that realizes $w$ in $\D_P(w)$,
and consider the corridor $B$ in $\D_P(w)$
about the trajectory $\tau$. We may take $B'\subset B$ so that $B'$ contains $\tau$ and so that
\[\left| (B'\cap P_i) -
v \left( P_i \right) \right| >\delta\]
for some $\delta >0$ and for all $i\in\{ 0,\ldots,n \}$, where $v(P)$ denotes the vertex set of a polygon $P$. That is, we may create a buffer $B'$ around $\tau$ so that within any development copy of $P$, $\tau$ is definitely at least $\delta$ away from every vertex of that particular copy. See the top of Figure \ref{fig:perturbation}.

\begin{figure}[!ht]
\centering
\includegraphics[width=4.5in]{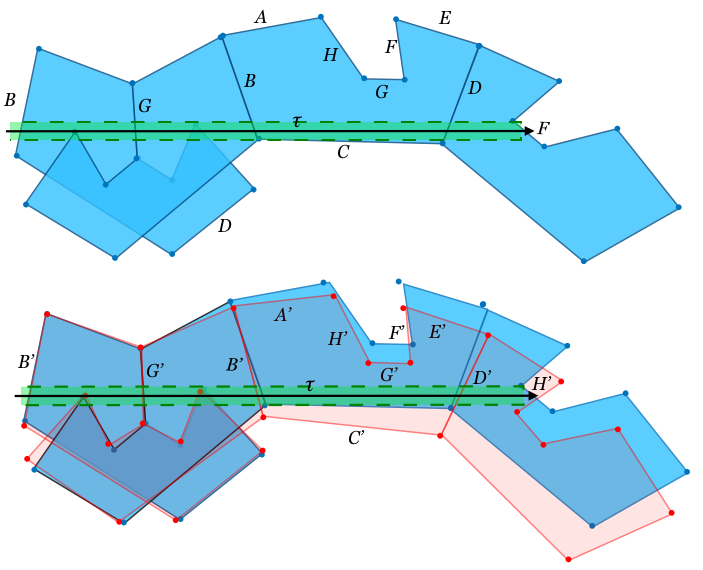}
\caption{Top: A trajectory (thick black) on a polygon $P$ (blue) realizing the finite bounce word \mbox{$w=BGBD$}. The four copies of $P$ shown comprise $\D_P(w)$. Bottom: A trajectory on a nearby polygon $P'$ (pink) that also realizes $w$. The four copies of $P'$ comprise $\D_{P'}(w)$. Note that the trajectory does not realize the longer word $BGBDF$ in $P'$; to realize this long word with this trajectory in a perturbation of $P$, a smaller value of $\varepsilon$ would have to be taken.}
\label{fig:perturbation}
\end{figure}

The coordinates of the vertices of each $P_i$ are functions of the coordinates of the vertices of $P$. More specifically, since the vertices of the $P_i$ are obtained by iterated reflections of the vertices of $P$, each coordinate of each vertex of $P_i$ is some affine function of the coordinates of the vertices of $P$, thought of as points $z_i = (x_i, y_i) \in \R^2$. In particular, all of these functions are continuous. This implies that there exists a value $\varepsilon$ so that varying the vertices of $P$ by less than $\varepsilon$ guarantees that the corresponding vertices of these perturbed $P'_i$ will vary from those of $P_i$ by less than $\delta$. In other words, there exists an $\varepsilon>0$ such that $P'_i \in (P_i)_\delta$ for all $P' \in P_\varepsilon$ (see the bottom of Figure \ref{fig:perturbation}).

Fix $B'$ in place and consider the families $P_\varepsilon$ and $P'_{\delta}$. Since each vertex of $P'_i$ is at most $\delta$ from its corresponding vertex in $P_i$, the vertices of $P'_i$ are disjoint from $B' \cap P'_i$ for all $i$. Further, when we restrict to any $P'_i$, $B'$ partitions the vertices of $P'_i$ into two sets --- those on one side of $\tau$ and those on the other --- in just the same way that $B'$ partitions the corresponding vertices of $P_i$.
This means that the trajectory $\tau$ within $B'$ hits the same sequence of edges in $\D_{P'}(w)$ as it does in $\D_P(w)$, namely $w$. Therefore $\tau$ also realizes $w$ in $P'$. Since this is true for every polygon in $P_\varepsilon$, the result follows.\end{proof}

We now apply Lemma \ref{lem:words_persist} to prove the main theorem of this section.

\begin{theorem}
\label{thm:no_finite}
Let $P$ be a polygon and $W$ a finite collection of finite words in its bounce language $\L_P$. Then there exists an $\varepsilon>0$ such that for all $Q\in P_\varepsilon$, $W\subset \L_Q$.
\end{theorem}

\begin{proof}
By Lemma \ref{lem:words_persist}, for each word $w_i \in W$ there is an $\varepsilon_i$ so that every $Q$ in $P_{\varepsilon_i}$ has $w_i$ in its bounce language. Let $\varepsilon$ equal the minimum of these $\varepsilon_i$. Then every $Q$ in $P_{\varepsilon}$ has $W$ as a subset of its bounce language.
\end{proof}

Rephrasing this in the language of the introduction, we have:

\begin{theorem}
A polygon $P$ cannot be reconstructed from any finite subset of its bounce language $\L_P$
\end{theorem}

\section{Future work: reconstructing lengths}\label{sec:lengths}

In \cite{DELS}, the authors show that $\B(P)$ is a complete invariant of  a polygonal billiard table, up to some elementary equivalence relations. Our Theorems \ref{thm:adj} and \ref{thm:hearing_angles} show how to reconstruct the adjacency of edges and sizes of angles of $P$ from $\B(P)$. To fully reconstruct $P$ from $\B(P)$, we would in addition need a method for reconstructing lengths of edges of $P$.

A na\"ive attempt to reconstruct lengths of a polygon from its bounce spectrum would be to simply measure the percentage of each letter (edge label) in the bounce spectrum, and then call this the length of the edge, where the polygon is normalized to have total perimeter 1.

\begin{question}\label{q:proportions}
Is it true that the length of each edge of a non-right-angled polygon, as a proportion of the total perimeter, is the same as the proportion of the corresponding edge label in the bounce spectrum?
\end{question}

As shown in Lemma \ref{lem:uncountable}, the bounce spectrum is a very large object, so it would be more desirable to be able to recover the lengths of edges from a \emph{single} bounce sequence. However, in general a single bounce sequence does not encode all edge length information. For example, if the table has rational angles, a trajectory travels in only finitely many directions, and this will bias the density of points hit on the edges. In a square billiard table, a trajectory of slope $\pi$ is aperiodic, equidistributes on the table, and hits each edge in a dense set of points. However, the trajectory hits the top and bottom edges $\pi$ times as often as the left and right edges, even though they have the same length.

It is possible that for irrational-angled billiard tables, computing these proportions for a single aperiodic trajectory would yield the length percentages.

\begin{question}
Consider a polygon with at least one irrational angle and an aperiodic trajectory that equidistributes on the billiard table. Let the coding of this trajectory be $(E_i)_{i \in \Z}$. Is the length of each edge of the polygon, as a proportion of the total perimeter, the same as the proportion of the corresponding edge label in $(E_i)$?
\end{question}

An explicit example or counterexample of such a table and trajectory would be interesting.

\newpage

\pagenumbering{roman}% for small Roman numerals
%\pagenumbering{arabic}% resets counter to 1
%\renewcommand*{\thepage}{A\arabic{page}}

\newpage

\appendix

\section{Another worked example}\label{app:effective_angles}
In Proposition \ref{prop:hear_3pi/2_gen}, we explicitly gave the bounce sequences for a vertex with angle $3\pi/2$. We will now do the same for angle $5\pi/12$. Figure \ref{fig:5pi12} shows the unfolding of this angle, which requires $2q=24$ copies of the corner. It also shows the $2p=10$ trajectories needed to characterize the angle, shown in blue for one direction and in red for the opposite direction.

\begin{figure}[!h]
\centering
\includegraphics[width=0.8\textwidth]{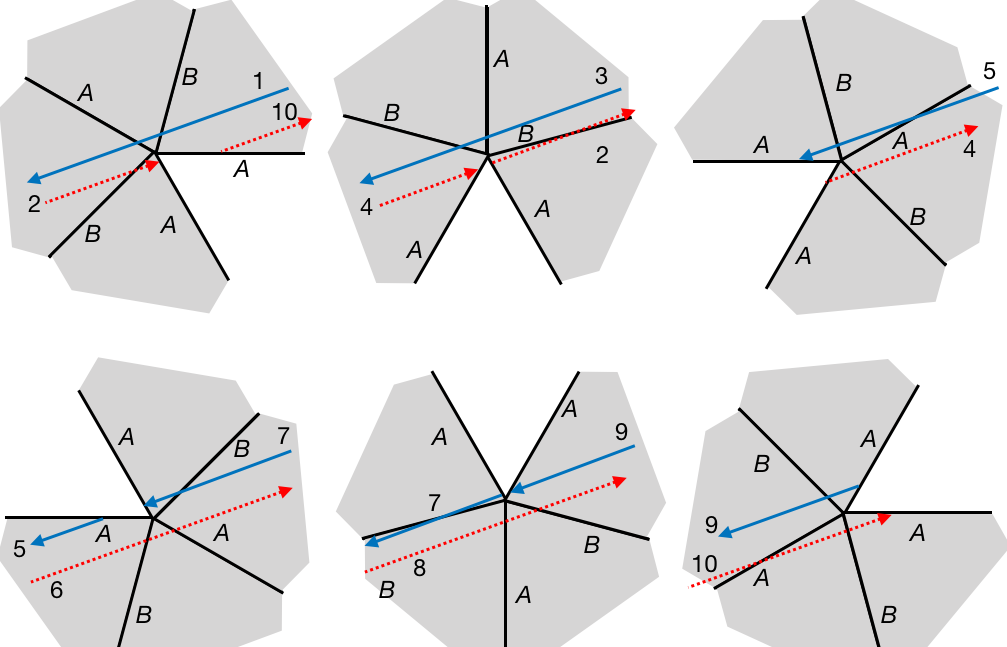}
\caption{An unfolded angle of $5\pi/12$ around the vertex, yielding $4\cdot 6 = 24$ copies for a total angle of $5\cdot 2\pi$. We use the ten nearby parallel trajectories $\tau_1, \ldots, \tau_{10}$ (blue for odd, red for even) and their associated bounce sequences to uniquely characterize the vertex angle as $5\pi/12$.}
\label{fig:5pi12}
\end{figure}

The bounce words corresponding to the $10$ limiting trajectories are given in \eqref{eq:5pi12}. Each of these words is in the bounce language $\L$, for all values of $n$.

\begin{equation}
\begin{split}\label{eq:5pi12}
  E^1_{-n}\  \ldots\ E^1_{-2}\  E^1_{-1}\  & \ \ B A\ \  \ E^1_1\  E^1_2\  \ldots\  E^1_n  \\
  E^1_n\  \ldots\ E^1_{2}\  E^1_{1}\  &BA B\  E^2_{-1}\  E^2_{-2}\  \ldots\  E^2_{-n} \\
  E^2_{-n}\  \ldots\ E^2_{-2}\  E^2_{-1}\  &\ \ AB\ \ \  E^2_1\  E^2_2\  \ldots\  E^2_n \\
  E^2_n\  \ldots\ E^2_{2}\  E^2_{1}\  &\ \ A B \ \ \ E^3_{-1}\  E^3_{-2}\  \ldots\  E^3_{-n}\\
  E^3_{-n}\  \ldots\ E^3_{-2}\  E^3_{-1}\  &A B A\  E^3_1\  E^3_2\  \ldots\  E^3_n \\
  E^3_n\  \ldots\ E^3_{2}\  E^3_{1}\  & \ \ B A\ \ \  E^4_{-1}\  E^4_{-2}\  \ldots\  E^4_{-n} \\
  E^4_{-n}\  \ldots\ E^4_{-2}\  E^4_{-1}\  &B A B \ E^4_1\  E^4_2\  \ldots\  E^4_n  \\
  E^4_n\  \ldots\ E^4_{2}\  E^4_{1}\  &\ \  A B\ \ \  E^5_{-1}\  E^5_{-2}\  \ldots\  E^5_{-n} \\
  E^5_{-n}\  \ldots\ E^5_{-2}\  E^5_{-1}\  &  \ \ AB\ \ \   E^5_1\  E^5_2\  \ldots\  E^5_n \\
  E^5_n\  \ldots\ E^5_{2}\  E^5_{1}\  &A B A\ E^1_{-1}\  E^1_{-2}\  \ldots\  E^1_{-n}\\
\end{split}
\end{equation}

\end{document}